\documentclass[10pt]{amsart}
\allowdisplaybreaks

\usepackage{amsfonts}
\usepackage{amsmath}
\usepackage{amssymb}
\usepackage{amsthm}
\usepackage{amsbsy}
\usepackage{xcolor}
\usepackage{graphicx} \usepackage{enumerate} \usepackage{multicol}
\usepackage{mathrsfs} \usepackage[all,cmtip]{xy}
\usepackage[color=blue!20,textsize=footnotesize]{todonotes}

\newcounter{dummy} \numberwithin{dummy}{section}
\newtheorem{theorem}[dummy]{Theorem}
\newtheorem{corollary}[dummy]{Corollary}
\newtheorem{lemma}[dummy]{Lemma}

\newtheorem{proposition}[dummy]{Proposition}
\theoremstyle{remark}
\newtheorem{remark}[dummy]{Remark}

\newcommand{\calL}{\mathcal{L}}
\newcommand{\calH}{\mathcal{H}}

\newcommand{\calX}{\mathcal{X}}

\newcommand{\scrF}{\mathscr{F}}

\newcommand{\scrO}{\mathscr{O}}

\newcommand{\frakT}{\mathfrak{T}}
\newcommand{\frakC}{\mathfrak{C}}
\newcommand{\frakE}{\mathfrak{E}}

\newcommand{\frakP}{\mathfrak{P}}

\newcommand{\frakg}{\mathfrak{g}}

\DeclareMathOperator{\id}{id}
\DeclareMathOperator{\inc}{inc}
\DeclareMathOperator{\pr}{pr}
\DeclareMathOperator{\rank}{rank}
\DeclareMathOperator{\spn}{span}
\DeclareMathOperator{\image}{image}
\DeclareMathOperator{\tr}{tr}

\DeclareMathOperator{\Hol}{Hol}

\DeclareMathOperator{\Ad}{Ad}

\DeclareMathOperator{\GL}{GL}

\DeclareMathOperator{\SO}{SO}

\DeclareMathOperator{\gr}{gr}

\DeclareMathOperator{\Gr}{Gr}

\DeclareMathOperator{\Lie}{Lie}
\newcommand{\Jac}{\rm Jac}
\newcommand{\stp}{{\bf stp}}

\DeclareMathOperator{\Iso}{Iso}
\DeclareMathOperator{\iso}{iso}

\DeclareMathOperator{\blambda}{\boldsymbol{\lambda}}

\newcommand{\Lbra}{[ \![}
\newcommand{\Rbra}{] \!]}

\newcommand{\ve}{\varepsilon}

\numberwithin{equation}{section}

\title[Effective normalization of sub-Riemannian connections]{Effective normalization of sub-Riemannian connections}
\author[E.~Grong, J.~Slovak]{Erlend Grong and Jan Slovak}

\address{University of Bergen, Department of Mathematics, P.O.~Box 7803, 5020 Bergen, Norway}
\email{erlend.grong@uib.no}

\address{Department of Mathematics and Statistics, Faculty of Science, Masaryk University,
Kotlářská 267/2, 611 37 Brno, Czech Republic}
\email{slovak@muni.cz}



\thanks{The first author is supported by the grant GeoProCo from the Trond Mohn Foundation - Grant TMS2021STG02 (GeoProCo). The second author acknowledges the support of the Czech Grant Agency project GA24-10887S, as well as the Horizon Europe MSCA projects CaLIGOLA and CaLiForNIA, ID 101086123 and ID 101119552}

\begin{document}

\begin{abstract}
We give a new normalization condition for connections on sub-Riemannian manifolds with constant symbols. The condition is formulated in terms of Cartan connections and depends only on the first degree of homogeneity of the curvature.
The essential part of our result is to show how a Cartan connection can be uniquely determined by a partial connection on the horizontal bundle. Viewed from the manifold, this observation is equivalent to the following claim: a compatible partial affine connection can be uniquely extended to both a full affine connection and a grading of the tangent bundle, and our normalization ensures that the holonomy of this connection will coincide with the horizontal holonomy, i.e., related to horizontal paths only. We give several examples in which we compute the canonical connections for a class of sub-Riemannian manifolds.
\end{abstract}

\maketitle


\section{Introduction}
Motivated by the equivalence problem for sub-Riemannian manifolds, we want to give a new normalization condition for connections.
Unlike what is the case for Riemannian manifolds and the Levi-Civita connection, there is no condition as simple as torsion freeness to use for normalization.
The first attempts of canonical connections were first provided for contact manifolds (see, e.g., \cite{Tan89}), before a general condition was given by Morimoto~\cite{Mor93,Mor08} using Cartan geometry, which corresponds to giving
a canonical condition for gradings on tangent bundles and compatible affine connections \cite{Gro20,Gro22}. These concern sub-Riemannian manifolds with constant symbols, i.e., spaces that have a constant flat space of reference that corresponds to the infinitesimal geometry in terms of the metric tangent cones, see \cite{Bel96} for precise statement.

Even though the condition from Morimoto is canonical, it is not necessarily simple to compute, and we seek to develop a different normalization condition in this paper. As movement in a sub-Riemannian manifold is restricted to a subbundle $E$ called \emph{the horizontal bundle}, a natural restriction would be to require that our connection is determined by parallel transport only in these distinguished directions. In \cite{FGR97,CGJK19}, holonomy of partial connections was treated, together with an explicit restriction of how connections can be extended while preserving the same holonomy group. Similar ideas are behind Rumin's differential complexes \cite{Rum90,Rum94,Rum99} where a version of the de Rham cohomology can be constructed from one-forms only defined on the horizontal bundles. A generalized framework in \cite{GrTr23} based on these ideas allows us to adapt these ideas to Cartan geometry. In the end, we provide a new normalization condition with the following proved benefits:
\begin{enumerate}[$\bullet$]
\item We introduce an extension condition so that we get a grading of the tangent bundle $TM$ and an affine connection $\nabla$ on $TM$, which are completely generated by the values $\nabla_{|E}|_E$ of the connection on the horizontal bundle.
\item We give a normalization condition for the curvature which only involves the part that has homogeneity degree~1, independent of the step of the horizontal bundle. In particular, this normalization condition determines $\nabla_{|E}|_E$, which in turn determines the grading and full connection.
\item Because of our extension condition, holonomy along all loops $\Hol_x(\nabla)$ and horizontal holonomy $\Hol_{E,x}(\nabla)$ along only loops tangent to $E$ will always coincide. In particular, if parallel transport along horizontal loops is trivial, then it will be trivial for all loops.
\end{enumerate}

Our result can be formulated in terms of Cartan geometry as follows. Let $(M, E,g)$ be a sub-Riemannian manifold with constant symbol being the Carnot algebra $\frakg_- =  \frakg_{-\stp} \oplus \frakg_{-\stp+1}\oplus\cdots \oplus \frakg_{-1}$, and the number $\stp$ is called the step of the manifold, see Section~\ref{sec:CarnotAlgebra} for definition of Carnot algebras and their isometries, and Section~\ref{sec:nonholonomic} for definition of sub-Riemannian manifolds with constant symbol. Let $G_0$ be the isometry group of $\frakg_-$, i.e., graded Lie algebra automorphisms restricting as isometries to $\frakg_{-1}$, and let its Lie algebra be $\frakg_0$. Define $\frakg = \frakg_- \oplus \frakg_0$ as a graded Lie algebra with its usual brackets. We can always use Lie brackets along with the inner product on $\frakg_{-1}$ to give an inner product on $\frakg_{-}$ and $\frakg$. Since the manifold has constant symbol, we can construct a principal bundle $G_0 \to \scrF \to M$ of nonholonomic frames, see again Section~\ref{sec:nonholonomic} for the definition. Cartan connections on $\scrF$ are $\frakg$-valued, $G_0$-equivariant one-forms $\psi$ (i.e., $r_g^*\psi = \operatorname{Ad}(g^{-1})\circ\psi$, where $r_g$ is the principal action of  $g\in\frakg_0$), recovering the fundamental vector fields (i.e., $\psi(\frac\partial{\partial t}_{|t=0}u\cdot\operatorname{exp}(ts)) = s$ for all $s\in\frakg_0$), whose restrictions to each of the tangent spaces are linear isomorphisms. The curvature $K = d\psi + \frac{1}{2} [\psi,\psi]$ of a Cartan connection $\psi$ gives us the function $\kappa \in C^\infty(\scrF, \frakg \otimes \wedge^2 \frakg_-^*)^{G_0}$, $\kappa(\psi(\cdot ), \psi(\cdot)) = K$,  where we impose our normalization condition. Each Cartan connection induces a filtration on $TM$ and Lie algebra structures on the individual graded tangent bundles. The Cartan connection is called adapted (or regular), if these two objects coincide with the standard ones on the filtered manifold $M$.  

On the complex $\frakC$ given by $\frakC^k =\frakg \otimes \wedge^k \frakg_-^*$, we introduce the Spencer differential $\partial$ and the base differential $\partial_b$ defined as the Lie algebra cohomology differential using respectively the adjoint representation and the trivial representation of~$\frakg_-$ on~$\frakg$. Let $\partial_b^{-1}$ be the pseudo-inverse defined using the inner product on $\frakC$. 
We then have the following result.

\begin{theorem} \label{th:main}
There is a unique adapted Cartan connection $\psi: T\scrF \to \mathfrak{g}$ with curvature $\kappa \in C^\infty(\scrF, \frakC^2)^{G_0}$ such that
$$\text{(Extension condition)} \qquad \partial_b^{-1} \kappa =0,$$
and its homogeneity degree one component $\kappa_1$ satisfies
$$\text{(Normalization condition)} \qquad \partial^* \kappa_1 = \partial^*\partial_b^{-1} \partial_b\kappa_1.$$
\end{theorem}
From the extension condition it follows that the Cartan connection $\psi$ is uniquely determined by its restriction $\psi|_{(\pi_*)^{-1} E}$ and the extension from $\psi|_{(\pi_*)^{-1} E}$ to $\psi$ is made in such a way that both objects have the same holonomy, see Lemma~\ref{lemma:CartanExtension} and Theorem~\ref{th:HolE} for details. Notice that the normalization condition deals only with the component of the curvature of homogeneity degree one, and our condition is clearly equivalent to the condition that $\partial^*\kappa_1$ vanishes on the kernel of the basic differential $\partial_b$. Any adapted Cartan connection can be viewed as an affine connection $\nabla$ on $TM$ and a decomposition $TM=TM_{-1} \oplus \cdots \oplus TM_{-\stp}$ such that $TM_{-1} = E$, $TM_{-j}=E^{-j}/E^{-j+1}$, where $E^{-j} = TM_{-1} \oplus \cdots \oplus TM_{-j}$ is spanned by $E$ and all brackets of $j$ vector fields with values in $E$.

We should emphasize that studying connections on sub-Riemannian manifolds from the point of view of symmetries, as we will do in this paper, will differ quite a lot from studying them from the point of view of geodesics. A natural comparison can be drawn to Riemannian reductive homogeneous spaces where the canonical connection compatible with isometries has torsion if the space is not symmetric, and hence not equal to the Levi-Civita connection, see, e.g., \cite{Agr06,OR12}. In sub-Riemannian geometry the difference is even more stark, as the natural structure for studying isomorphism classes are canonical affine connections and gradings, while the most natural framework for studying geodesics and Jacobi fields are non-linear connections, see, e.g., \cite{LiZe09,ABR18,Gro20b}.

The structure of the paper is as follows. In Section~\ref{sec:Complex} we look at Carnot algebras and complexes of linear forms on these spaces. In Section~\ref{sec:ConnectionSR} we look at sub-Riemannian manifolds of constant symbol and connections that are compatible with all of these structures in that they preserve them under parallel transport. We will present this theory both from the point of view of Cartan geometry and in terms of gradings of the tangent bundle and affine connections. In Section~\ref{sec:ExtendPartial} we go into more details of how a partial connection can be extended in a unique way. Finally, we will prove Theorem~\ref{th:main} and go into more details on the normalization condition in Section~\ref{sec:Normalization}. The remaining parts of the paper consist of worked-out examples. Section~\ref{sec:Small} gives an explicit description of how a connection can be computed when the growth vector of the horizontal bundle is $(2,3)$ or $(2,3,5)$. In Section~\ref{sec:Contact} we go into details on contact manifolds, while step two sub-Riemannian manifolds with the maximal growth vector $(n_1, n_1(n_1+1)/2)$ are treated in Section~\ref{sec:MaxGrowth}.

\subsection{Notation and conventions} \label{sec:Notation}
We will use the following notation for the rest of the paper. If $\ell: V \to W$ is a linear map between two (finite dimensional) inner product spaces, we will use $\ell^*: W \to V$ for its adjoint with respect to the inner products, which will appear most frequently in our case. For the rare cases when we actually need the dual map, we will use $\ell^\vee: W^* \to V^*$ to distinguish them. We will also often use the notation $\ell^{-1}:W \to V$ for the pseudo-inverse, defined so that $\ell^{-1}$ vanishes on $(\image \ell)^\perp$, while being equal to the inverse of $\ell |_{(\ker \ell)^{\perp}}$ on $\image \ell$. We note that $\ell^{-1}$ and $\ell^*$ coincide when $\ell$ is an isometry from $(\ker \ell)^\perp$ to $\image \ell$, but even when this is not the case, the two maps have the same image and kernel. Furthermore, if $\ell$ is a surjective map, then the definition $\ell^{-1}$ does not depend on the inner product of $W$. Hence, if $\ell: V \to W$ is a surjective map from an inner product space $V$ to a vector space $W$, then $\ell^{-1}$ is still well defined and we can introduce an inner product on $W$ by
$$\langle w_1, w_2 \rangle = \langle \ell^{-1} w_1, \ell^{-1} w_2 \rangle, \qquad w_1, w_2 \in W.$$
We call this inner product on $W$ \emph{induced} by $\ell$.

If $v \in V$ is an element in an inner product space, we will use $v^* = \langle v, \cdot \rangle \in V^*$. If $v_1, \dots, v_n$ is any basis for a vector space, with or without an inner product, we denote $v_1^*, \dots, v_n^*$ for the dual basis given by $v_i^*(v_j) = \delta_{ij}$.

\section{Complexes of Carnot algebras} \label{sec:Complex}
\subsection{Carnot algebras and isometries} \label{sec:CarnotAlgebra}
Let $\mathfrak{g}_- = \mathfrak{g}_{-\stp} \oplus \cdots \oplus \mathfrak{g}_{-1}$ be a negatively graded Lie algebra. We assume that the Lie bracket is homogeneous of degree zero, implying in particular that $\mathfrak{g}_-$ is nilpotent. We say that the Lie algebra $\mathfrak{g}_-$ is \emph{stratified} if $[\mathfrak{g}_{-1}, \mathfrak{g}_{-j}] = \mathfrak{g}_{-j-1}$ for any $1 \leq j < \stp$. \emph{A Carnot algebra} is a negatively graded, stratified Lie algebra $\mathfrak{g}_-$ with a chosen inner product $\langle \cdot , \cdot \rangle$ on~$\mathfrak{g}_{-1}$.  \emph{Isometries} of a Carnot algebra $\mathfrak{g}_-$ to itself are graded Lie algebra automorphisms that will map $\mathfrak{g}_{-1}$ to itself isometrically. The set of all isometries $G_0$ forms a Lie group~\cite{LDO16}. Let $\mathfrak{g}_0$ be its Lie algebra and define $\mathfrak{g} = \mathfrak{g}_- \oplus \frakg_0$ with the induced brackets,
$$[A_1 +s_1,  A_2 + s_2] =  [A_1,A_2] + s_1(A_2) - s_2(A_1) +s_1 s_2 - s_2 s_1,$$
with $s_1, s_2 \in \mathfrak{g}_0$, $A_1, A_2 \in \mathfrak{g}_{-}$. We note that $\mathfrak{g}_0$ consists of inner Lie algebra derivations that are skew-symmetric when restricted to $\mathfrak{g}_{-1}$.

The inner product on $\mathfrak{g}_{-1}$ also allow us to define a product on $\mathfrak{g}$ as follows. By the definition of a Carnot algebra, the map $A \wedge B \mapsto [A,B]$ is a surjective map from elements in $(\wedge^2 \frakg_-)_{-j}$ of degree $-j$ to $\frakg_{-j}$ (of course we consider $\frakg_{-j}=0$ for all $j>\stp$). We can introduce an inner product on $(\wedge^2 \frakg_{-})_{-2} = \wedge^2 \frakg_{-1}$ by using one half of the inner product induced on $\frakg_{-1} \otimes \frakg_{-1}$ so that the wedge of two unit orthogonal elements in $\frakg_-$ have length 1. We can then define an inner product on $\frakg_{-2}$ induced by $[\cdot, \cdot]$ as described in Section~\ref{sec:Notation}. Repeating this argument, we obtain an inner product for the whole of $\frakg_-$. We can further give $\frakg_0$ an inner product by considering it as a subspace of $\frakg_{-} \otimes \frakg^*_-$.

\subsection{Complexes of linear forms}
Let $\frakC^k = \frakg \otimes \wedge^k \mathfrak{g}_{-}^* $ be linear forms on $\mathfrak{g}_-$ to~$\mathfrak{g}$, inheriting the grading from $\mathfrak{g}$ in the usual way.
We will consider three types of operations on this complex.
\begin{enumerate}[$\bullet$]
\item \emph{Wedge products:} If $\alpha \in \frakC^k$ and $\beta \in \wedge^j \frakg_-^*$, then we can define $\alpha \wedge \beta \in \frakC^{k+j}$, by
$$\alpha \wedge \beta = \sum_{l=1}^r A_l \otimes (\alpha_l \wedge \beta), \qquad \alpha = \sum_{l=1}^r A_l \otimes \alpha_l \qquad A_l \in \frakg , \alpha_l \in \wedge^k \frakg_-^*.$$
In particular, since $\frakg = \frakC^0$, we can write $\alpha =  \sum_{l=1}^r A_l \otimes \alpha_l = \sum_{l=1}^r A_l \wedge \alpha_l$.
\item \emph{Lie brackets:} If $\alpha \in \frakC^k$ and $\beta \in \frakC^j$, then we can define $[\alpha,\beta] \in \frakC^{k+j}$ such that if $\alpha = \sum_{l=1}^r A_l \wedge \alpha_l$ and $\beta = \sum_{m=1}^q B_m \wedge \beta_m$, with $A_l, B_m \in \frakC^0$, $\alpha_l \in \wedge^k \frakg_-^*$, $\beta_m\wedge^j \frakg_-^*$, then
$$[\alpha,\beta] = \sum_{l=1}^r\sum_{m=1}^q [A_l,B_m] \wedge (\alpha_l \wedge \beta_m)$$
Remark that if $\alpha \in \frakC^1$, then $[\alpha,\alpha](A,B) = 2[\alpha(A),\alpha(B)]$ for any $A, B \in \frakg_-$.
\item \emph{Differentials:} Let $\xi$ be a representation of $\frakg_-$ on $\frakg$ such that the map $(A,B) \mapsto \xi(A)B$ has degree zero. The corresponding Lie algebra differential is given by
\begin{align*}
(\partial_\xi \alpha)(A_0, \dots, A_k) & =\sum_{i=0}^k (-1)^i \xi(A_i) \alpha(A_0, \dots, \hat A_i, \dots, A_k) \\
& \qquad + \sum_{i<j} (-1)^{i+j} \alpha([A_i, A_j], A_0, \dots, \hat A_i, \dots, \hat A_j, \dots, A_k),
\end{align*}
which makes $\partial_\xi$ a degree zero operator that satisfies $\partial_\xi^2 = 0$, making $(\frakC,\partial_\xi)$ a graded differential complex. We will consider two differentials in particular, \emph{the Spencer differential} $\partial$ and \emph{the base differential} $\partial_b$ corresponding to respective choices of $\xi$ being equal to the adjoint and the trivial representations.
\end{enumerate}
Let us give some identities combining these concepts. We can extend the definition $\partial_b$ to $\wedge \frakg_-^*$ using the trivial representation of $\frakg_{-1}$ on $\mathbb{R}$, giving us the identity
$$\partial (\alpha \wedge \beta) = (\partial \alpha) \wedge \beta + (-1)^k \alpha \wedge \partial_b \beta, \qquad \alpha\in \frakC^k, \beta \in \wedge^j \frakg_-^*,$$
and similarly $\partial_b (\alpha \wedge \beta) = (\partial_b \alpha) \wedge \beta + (-1)^k \alpha \wedge (\partial_b \beta)$. Also observe that for $\beta \in \wedge^j \frakg_{-}^*$, $A \in \frakg_-$, $B \in \frakg = \frakC^0$, then
$$\partial B(A) = [A,B], \qquad \partial_b B = 0, \qquad (\partial - \partial_b) (B \wedge \beta) = (\partial B) \otimes \beta.$$
Furthermore, if we consider the identity map $\id = \id_{\frakg_-}$ as an element in $\frakC^1$, then
$$(\partial - \partial_b) \alpha = [\id, \alpha], \qquad \alpha \in \frakC.$$

\begin{remark} \label{re:No closed one-forms}
We will need the following fact for the proof of our main result. Observe that $\frakC_1^1 = \oplus_{i=1}^{\stp} (\frakg_{-i+1} \otimes \frakg_{-i}^*)$ are the degree 1 elements in $\frakC^1$.
From \cite[Proposition~1]{Mor08}, we know that the Tanaka prolongation of $\frakg$ coincides with $\frakg$, meaning in particular that $\frakg^1 = \{ \alpha \in \frakC_1^1 \, : \, \partial \alpha =0\} = \ker \partial \cap \frakC_1^1 =0$, see, e.g., \cite{Zel09} for details.
\end{remark}

\subsection{Filtration of linear forms}
Define $\partial_b^{-1}$ as the pseudo-inverse of $\partial_b$ with respect to the inner product on $\frakC$. Since $\image \partial_b \subseteq \ker \partial_b$, we can get the same relation for $\partial_b^{-1}$ using orthogonal complements and hence $(\partial_b^{-1})^2 =0$.
Using the base differential $\partial_b$ as well as the inner product on $\frakC$, we can define an orthogonal decomposition
$$\frakC = \image(\partial_b) \oplus_\perp \image(\partial_b^*) \oplus_\perp \ker(\Box_b), \qquad \Box_b = \partial_b \partial_b^* + \partial_b^* \partial_b.$$
If $\frakT = (\ker \partial_b)^\perp = \image \partial_b^* = \image \partial_b^{-1}$ and we write $\frakE_0 = \ker(\Box_b) = \ker(\partial_b) \cap \ker(\partial_b^{-1})$, then $\image \partial_b = \partial_b \frakT$, and so we can write
$$\frakC = \frakT \oplus_{\perp} \partial_b \frakT \oplus_\perp \frakE_0.$$
The orthogonal projection to $\frakE_0$ is given by
$$\Pi = \id- \partial_b^{-1} \partial_b - \partial_b \partial_b^{-1}.$$
While $\frakE_0$ is not closed under $\partial$, there is a one-to-one correspondence with a subcomplex $\frakP$ in $\frakC$ defined as follows.

We introduce the following filtration $\frakC_{(0)} \supsetneq \frakC_{(1)}  \supsetneq \cdots \supsetneq \frakC_{(S)} \supsetneq \frakC_{(S+1)} =0$ of finite depth $S$ such that
$$\frakC_{(j)} = \frakg \otimes (\wedge \frakg_{-}^*)_{\geq j},$$
or equivalently
$$\frakC_{(j)} = \{ \alpha \in \frakC\, : \,  \text{$\alpha$ vanishes on elements of degree $\leq -j+1$} \}.$$
The depth $S$ equals the degree of nonzero elements in $\wedge^n \frakg_{-}^*$, $n = \dim \frakg_-$, which is the maximal possible. We remark that by definition $\frakC^j \subseteq \frakC_{(j)}$, $\frakC^j \cap \frakC_{(j \cdot \stp+1)} =0$ and furthermore $\partial \frakC_{(j)} \subseteq \frakC_{(j)}$. Since $$(\partial-\partial_b) \frakC_{(j)} \subseteq \frakC_{(j+1)},$$
we have that $\partial_b$ is a base differential of the filtered complex is the sense of \cite{GrTr23}, giving us the following results. 

Introduce a projection operator $P^\infty$ as follows. Consider
$$P = \id - \partial_b^{-1} \partial - \partial \partial_b^{-1}, $$
notice $P^{S+1} = P^S$ and write $P^\infty =P^S$. 
If 
$$\frakP = \image P^\infty = \ker \partial_b^{-1} \cap \ker \partial_b^{-1} \partial = \{ P(\alpha) = \alpha \, : \,\alpha \in \frakC \}$$
then
\begin{equation} \label{CDecomp} \frakC = \frakT \oplus \partial \frakT \oplus \frakP = \frakT \oplus \partial_b \frakT \oplus \frakP, \qquad \text{with $\ker P^\infty = \frakT \oplus \partial \frakT$} .\end{equation}
Furthermore, $\partial P^\infty = P^\infty \partial$ and $P^{\infty}$ is an isomorphism from $\frakC$ to $\frakP$ on the level of cohomology. Finally, if $\alpha \in \frakC^k$, and $S_k \leq k\cdot s$ is the maximal number such that $\frakC^k \cap \frakC_{(S_k)}$ is nonzero, then $P^\infty(\alpha) = P^{S_k-k}(\alpha)$.

The decompositions in \eqref{CDecomp} are not orthogonal in general, however, by the proof of \cite[Proposition 2.9]{GrTr23},
$$\text{$P^\infty \circ \Pi = P^\infty$ and $\Pi \circ P^\infty = \Pi$},$$
giving us a one-to-one correspondence between $\frakE_0$ and $\frakP$. Explicitly, the correspondence is given in the following way.
\begin{lemma} \label{lemma:Pinfty}
Let $\alpha \in \frakC$ and $\beta = P^\infty \alpha$. Then $\beta$ is the unique solution to
$$\Pi \alpha= \Pi\beta, \qquad \partial^{-1}_b \partial \beta =0, \qquad \partial^{-1}_b \beta =0.$$
In particular, if $\alpha \in \frakC^1$, then $\beta$ is the unique solution to
$$\alpha|_{\frakg_{-1}} = \beta|_{\frakg_{-1}}, \qquad \partial^{-1}_b \partial \beta =0.$$
\end{lemma}
\begin{proof}
We first note that $\Pi\beta = \Pi P^\infty \alpha = \Pi \alpha$. Furthermore, since $\beta \in \frakP$, we have $\partial_b^{-1} \partial \beta =0$ and $\partial_b^{-1} \beta =0$. To show that these equations uniquely determine~$\beta$, let~$\beta_2$ be a second solution. Then $\gamma = \beta - \beta_2 \in \frakP$, with $\gamma \in \frakP$ and $\Pi \gamma =0$. However, then
$$0  = P^\infty 0 = P^\infty\Pi \gamma = P^{\infty}\gamma = \gamma,$$
giving us uniqueness.
\end{proof}

\begin{remark} \label{re:Boxb}
\begin{enumerate}[\rm (a)]
\item We observe that by the definition of the Carnot algebra, the image of the Lie bracket is $\mathfrak{g}_{\leq -2}$. It follows that
$$\partial^*_b \frakC^2 = \partial_b^{-1} \frakC^2 = \{ \alpha \in \frakC^1 \, : \, \alpha |_{\frakg_{-1}} =0 \}.$$
Since $\image(\partial_b) \cap \frakC^1 =0$, we have that $\ker(\Box_b) \cap \frakC^1 = \frakg_{-1}^* \otimes \mathfrak{g}$ .
Furthermore, $\ker(\Box_b) \cap \frakC^1 = \ker(\partial_b) \cap \frakC^1$, so $\ker(\partial_b) \cap \frakC^1  = \{ \alpha \in \frakC^1 \, : \, \alpha |_{\frakg_{\leq -2 }} =0 \}$.
\item From the way we defined the inner product on $\frakg$ in Section~\ref{sec:CarnotAlgebra}, for any $\alpha_1, \alpha_2 \in \frakg \otimes \frakg_{\leq -2}$,
$$\langle \partial_b \alpha_1, \partial_b \alpha_2 \rangle = \langle \alpha_1, \alpha_2 \rangle.$$
Hence, for $\beta \in \frakC^2$, we have $\partial_b^{-1} \beta = \partial_b^* \beta$.
\end{enumerate}
\end{remark}

\section{Connections on sub-Riemannian manifold with constant symbol} \label{sec:ConnectionSR}

\subsection{The nonholonomic frame bundle} \label{sec:nonholonomic}
A sub-Riemannian manifold is a triple $(M,E,g)$, where $E$ is a subbundle of the tangent bundle $TM$ of a connected manifold $M$ and $g= \langle \cdot , \cdot \rangle_g$ is a metric tensor on $E$. We will assume that $E$ is a \emph{bracket-generating, equiregular subbundle}, meaning that there is a flag of subbundles
\begin{equation} \label{flag} 0 = E^0 \subsetneq E= E^{-1}  \subsetneq E^{-2} \subsetneq \cdots \subsetneq E^{-\stp} = TM,\end{equation}
where
$$E^{-(j+1)}= \spn \{ X, [X,Y] \, : \, X \in \Gamma(E^{-j}) , Y \in \Gamma(E)\}.$$
We write $n_j = \rank E^{-j} - \rank E^{-j+1}$ and let $(n_1, n_1 +n_2, \dots, n_1+ \dots + n_{\stp})$ be \emph{the growth vector} of $E$.
\emph{The nonholonomic tangent bundle $\gr_- = \gr_-(M,E)$} of $(M,E)$ is defined by
$$\gr_{-} = \gr_{-\stp} \oplus \cdots \oplus \gr_{-1}, \qquad \gr_{-j} = E^{-j}/E^{-j+1}.$$
We note that if we define a bracket $\Lbra \cdot , \cdot \Rbra$ on $\gr_-$ by
$$\Lbra X +E^{-i},  Y +E^{-j} \Rbra = [X,Y] + E^{-(i+j)+1}, \qquad X \in \Gamma(E^{-i-1}), Y \in \Gamma(E^{-j-1}),$$
then this bracket is tensorial. It follows that $(\gr_{x,-}, \Lbra \cdot , \cdot \Rbra)$ is a Lie algebra for any $x \in M$, which by definition is nilpotent and stratified. If we define the inner product $\langle \cdot, \cdot \rangle_{g(x)}$ on $\gr_{x,-1} =E_x$, then $\gr_{x,-}$ has the structure of a Carnot algebra. We say that $(M,E,g)$ has \emph{a constant symbol} $\mathfrak{g}_-$ if each $\gr_{x,-}$, $x\in M$, is \emph{isometric} to the same Carnot algebra $\mathfrak{g}_-$. For the remainder of this paper, we will only consider sub-Riemannian manifolds with constant symbol. We emphasize that all Carnot algebras $\gr_{x,-}$, have to be isometric, not merely isomorphic as Lie algebras.

For a sub-Riemannian manifold $(M,E,g)$ with constant symbol $\mathfrak{g}_-$, we say that \emph{a nonholonomic orthonormal frame} at $x \in M$ is a Carnot algebra isometry $u: \mathfrak{g}_{-} \to \gr_{x,-}$. We write $\scrF_x$ for the set of all such frames at $x$. If $G_0$ is the isometry group of $\mathfrak{g}_-$, then there is the $G_0$-principal bundle
$$G_0 \to \scrF \stackrel{\pi}{\to} M, $$
with fiber $\scrF_x$ over $x \in M$ and where $g_0 \in G_0$ acts on $u \in \scrF_x$ by $u \cdot g_0 = u \circ g_0$.

\begin{remark}
For a general growth vector, there mostly exist many non-isomorphic Carnot algebras, so the restriction of having a constant symbol is usually quite serious. 
Any free nilpotent algebra of a given step $\stp$ is an exception. This is the truncated tensor algebra up to step $\stp$ divided out by the skew-symmetry and Jacobi identity of the Lie bracket. It is the maximal growth vector and there is a unique Carnot algebra with this property. Thus, all sub-Riemannian manifolds with maximal growth vectors will have constant symbols.
\end{remark}

\subsection{Cartan connections on the nonholonomic bundle} 
Let $\GL(\mathfrak{g})$ be invertible linear maps of $\mathfrak{g}$ and write $\hat \frakC^k = \frakg \otimes \wedge^k \mathfrak{g}^* $. We define the action $\rho(\ell)$ of $\ell \in \GL(\mathfrak{g})$ on $\alpha \in \hat \frakC^k$ by
$$\rho(\ell)\alpha(A_1, \dots, A_k) = \ell \alpha(\ell^{-1} A_1, \dots, \ell^{-1}A_k), \qquad A_1, \dots, A_k \in \mathfrak{g}.$$
If $\alpha \in \frakC^k$, then we can consider it as an element in $\hat \frakC^k$ by requiring it to vanish on~$\mathfrak{g}_0$. Furthermore, if $\ell(\mathfrak{g}_0) = \mathfrak{g}_0$, then $\rho(\ell)\alpha$ is well defined as an element in $\frakC^k$.

Let $(M, E,g)$ be a sub-Riemannian manifold with constant symbol $\mathfrak{g}_-$. Let $\mathfrak{g}_0$ be the isometry algebra of $\mathfrak{g}_{-}$ and write $\mathfrak{g} = \mathfrak{g}_{-} \oplus \frakg_0$. Let $\psi: T\scrF \to \mathfrak{g}$ be a Cartan connection on the non-holonomic frame bundle. Recall that this is a $\frakg$-valued one-form, such that $\psi_u:T_u \scrF \to \frakg$ is always a linear isomorphism, it satisfies equivariance property $\psi(w \cdot g_0) = \Ad(g_0^{-1}) \psi(w)$ and satisfies $\psi(\frac{d}{dt} u \cdot e^{ts}|_{t=0}) = s$ for $u \in \scrF$, $w \in T_u \scrF$, $g_0 \in G_0$, $s \in \frakg_0$.

We write it as $\psi = ( \theta,\omega)$ where $\theta$ and $\omega$ have values in respectively $\mathfrak{g}_-$ and $\mathfrak{g}_0$. Introduce the notation $\hat E^{-j} = \pi_*^{-1} E^{-j}$ and remark that then $\hat E^0 = \ker \pi_*$ is the vertical bundle. Observe then by definition $\hat E^{-k}$ is spanned by $\hat E^{-1}$ and brackets of its section up to $k$ elements. We say that $\psi$ is \emph{adapted} or \emph{regular} if for any $w \in \hat E^{-j-1}_u \subseteq T_u \scrF$, $u \in \scrF$,
$$u^{-1} (\pi_* w +E^{-j}) = \psi(w) \bmod \mathfrak{g}_{\geq -j} = \theta(w) \bmod \mathfrak{g}_{\geq -j}.$$
To understand this condition, let us first consider the Riemannian special case when $\stp =1$. Then $\gr(M,E) = TM$, and being adapted means that for any $w \in T_u\scrF$, $\theta(w)$ equals $\pi_* w$ expressed in the frame $u$. For higher step, $u$ is a nonholonomic frame, and we want the part of $\theta(w)$ of minimal homogeneity to express $\pi_* w+E^{-j}$ in $u$.

As usual, we define the curvature $\kappa = \kappa^\psi \in C^\infty(\scrF, \frakC^2)^{G_0}$ of $\psi$ by
$$K = d\psi + \frac{1}{2} [\psi, \psi], \qquad \kappa(\psi(v),\psi(w)) = \kappa(\theta(v), \theta(w)) = K(v,w).$$
Furthermore, introduce the operator $d^{\psi}: C^\infty(\scrF, \frakC^k)^{G_0} \to C^\infty(\scrF, \frakC^{k+1})^{G_0}$ by
\begin{align*}
d^\psi\alpha(u)(A_0, \dots, A_k) & = \sum_{j=0}^k (-1)^j d\alpha(u)(\psi^{-1} A_j)(A_0, \dots, \hat A_j, \dots, A_k). 
\end{align*}
Observe that the operator $d^\psi$ has only terms of positive homogeneity. Hence, $(d^{\psi}\alpha)_k$ will only depend on the lower homogeneity components $\alpha_i$, $i<k$.

\begin{lemma} \label{lemma:LowerHom}
Let $\psi$ and $\tilde \psi$ be two adapted connections with respective curvatures $\kappa$ and $\tilde \kappa$. Then their curvatures have only terms of strictly positive homogeneity.

Furthermore, we can write $\psi = (\id +\alpha) \tilde \psi$, where $\alpha\in C^\infty(\scrF, \frakC^1)^{G_0}$ only has terms of strictly positive homogeneity, and
\begin{equation} \label{KappaWithAlpha} \kappa = \rho(\id+\alpha)\left( (\id + \alpha)^{-1} d^{\tilde \psi}\alpha+ \tilde \kappa - \frac{1}{2}[\id, \id]\right) + \frac{1}{2}[\id ,\id].\end{equation}
In particular, there are terms $\ve_k = \ve_k(\alpha_1, \dots, \alpha_{k-1},(d^{\tilde \psi} \alpha)_2, \dots, (d^{\tilde \psi} \alpha)_k)$ only depending on terms of $\alpha$ of lower homogeneities than $k$ such that
\begin{equation} \label{KappaDegree} \kappa_k = \partial \alpha_k + \ve_k.\end{equation}
For the case $k=1$, we have $\kappa_1 = \partial \alpha_1 + \tilde \kappa_1$.
\end{lemma}
Remark that if we define $\alpha^{\circ j} = \alpha \circ \alpha \circ \cdots \circ \alpha$ 
as the result of composing $\alpha$ with itself $j$-times, with the convention that it vanishes at $\mathfrak{g}_0$, then $\alpha^{\circ (\stp + 1)} = 0$ since $\alpha$ only has positive terms. As a consequence
\begin{equation}\label{AlphaInv} (\id+\alpha)^{-1} = \sum_{j=0}^{\stp} (-1)^j \alpha^{\circ j}, \end{equation}
so in particular $\id + \alpha$ is an invertible map and $\rho(\id +\alpha)$ is well-defined.
\begin{proof}
For an element $A \in \mathfrak{g}_-$, let $H_A$ be the vector field on $\scrF$ with $\psi(H_A) = A$. Each adapted connection $\psi$ 
has got the property $\hat E^{-j-1} = \hat E^{-j} \oplus \{ H_A \, : \, A \in \mathfrak{g}_{-j-1}\}$. It follows that if $A \in \mathfrak{g}_{-i}$ and $B \in \mathfrak{g}_{-j}$, then $[H_{A}, H_B]$ takes values in $\hat E^{-i-j}$. Furthermore, since
$$K(H_A, H_B) = - \psi([H_A, H_B] )+ [A,B],$$
and $\psi$ is adapted, we conclude $K(H_A,H_B) = 0 \bmod \frakg_{-i-j+1}$. It follows that the curvature only has terms of positive homogeneity. 
If~$\tilde H_A$ is the vector field that satisfies $\tilde \psi(\tilde H_A) = A$, $A \in \mathfrak{g}_{-j}$, then since both connections are adapted $\psi(\tilde H_A) = (\psi \circ \tilde \psi^{-1})(A) = A \bmod \mathfrak{g}_{-j+1}$, and $\alpha$ can only have positive degrees.

Since $\psi = (\id  + \alpha) \tilde \psi$, we obtain
\begin{equation}\label{eqkappa}
    K  =  d\alpha \wedge \tilde \psi + (\id+\alpha) d\tilde \psi+  \frac{1}{2} [(\id+\alpha) \tilde \psi, (\id + \alpha) \tilde \psi],
\end{equation}    
and $\kappa(u)(A,B) = K(H_A(u),H_B(u))$. Next, observe 
$H_A = \tilde H_{(\id+\alpha)^{-1}A}$
and compute
\begin{align*}
    (d\alpha \wedge \tilde \psi)&(H_A, H_B) = (H_A\alpha)(\tilde\psi(H_B)) - (H_B\alpha)(\tilde\psi(H_A))\\
    &=(\tilde H_{(\id+\alpha)^{-1}A} \alpha)((\id+\alpha)^{-1} B) - (\tilde H_{(\id+\alpha)^{-1})B} \alpha)((\id+\alpha)^{-1} A) \\
    & = (\rho(\id+\alpha) ((1+\alpha)^{-1} d^{\tilde \psi}\alpha)(A,B).
\end{align*}
Plugging this into \eqref{eqkappa}, and adjusting the remaining terms, we verify \eqref{KappaWithAlpha}.

To obtain \eqref{KappaDegree}, we use \eqref{AlphaInv} to compute
$$\left(-\frac{1}{2} \rho(\id+\alpha)[\id, \id] + \frac{1}{2} [\id, \id]  \right)_k = [\id , \alpha_k] - \frac{1}{2} \alpha_k[\id, \id ] + \text{terms of lower degree},$$
and $[\id , \alpha_k] - \frac{1}{2} \alpha_k[\id, \id ]$ is exactly $\partial \alpha_k$. The final observation is that the degree $k$ component of $\rho(1+\alpha)\left( (1-\alpha)^{-1} d^{\tilde\psi}\alpha+ \tilde\kappa\right)$ and its derivatives only depends on $\alpha_1, \dots, \alpha_{k-1}$ and its derivatives.
\end{proof}

\begin{remark}
The observation in \eqref{KappaDegree} along with the result in Remark~\ref{re:No closed one-forms} is the background for Morimoto's result in \cite{Mor93,Mor08} stating that every sub-Riemannian manifold with constant symbol has an adapted Cartan connection satisfying
$$\partial^*  \kappa = 0.$$
See examples where these connections are explicitly computed in \cite{AMS19,Gro20,Gro22}.
\end{remark}

\begin{remark}
We will give a further explicit presentation of the relation in Lemma~\ref{lemma:LowerHom}. Define $(\id + \alpha)^{-1} = \id + \alpha_{inv}$, where $\alpha_{inv} = \sum_{j=1}^s (-1)^j \alpha^{\circ j}$. Then
\begin{align} \label{secondKappa}
\kappa(A,B) & = - \alpha([A,B])+ (d^{\tilde\psi}\alpha + (\id + \alpha)\tilde \kappa)(A + \alpha_{inv}(A) , B + \alpha_{inv}(B)) \\ \nonumber
& \qquad +(\id+\alpha) \left( - [\alpha_{inv}(A), B]   +[\alpha_{inv}(B),A]  - [ \alpha_{inv}(A) ,  \alpha_{inv}(B)]\right) .
\end{align}
\end{remark}

\begin{remark}
Recall the Bianchi identity
$$dK + [\psi , K] =0.$$
Observe that the Bianchi identity can be rewritten as
$$d^\psi \kappa + \circlearrowright \kappa(\kappa(\cdot, \cdot ), \cdot) + \partial \kappa =0,$$
see \cite[Chapter~1.5]{CaSl09} for details. Here and elsewhere, we use $\circlearrowright$ for the triple cyclic sum,
$$\circlearrowright \kappa(\kappa(\cdot, \cdot ), \cdot)(A,B, C) = \kappa(\kappa(A, B), C) + \kappa (\kappa(C,A), B) + \kappa(\kappa(B, C), A) $$
Looking at this equation in first homogeneity, we obtain the important relation
\begin{equation} \label{kappa1closed} \partial \kappa_1 =0.\end{equation}
\end{remark}

\subsection{Cartan connections seen from the manifold}
We use approach from \cite{Gro20,Gro22}.
Let $(M, E, g)$ be a sub-Riemannian manifold with $E$ equiregular and corresponding growth flag as in \eqref{flag}. \emph{An $E$-grading} is a vector bundle isomorphism $I: TM \to \gr_-$ such that
$$I(E^{-j}) = \gr_{-j} \oplus \cdots \oplus \gr_{-1}.$$
Such an $E$-grading is equivalent to defining a grading $TM_{-1} \oplus TM_{-2} \oplus \cdots \oplus TM_{-\stp}$ such that $E^{-j} = TM_{-1} \oplus \cdots \oplus TM_{-j} = E^{-j+1} \oplus TM_{-j}$, where the correspondence is given by $TM_{-j} = I^{-1} \gr_{-j}$.
Note that on $\gr_{-}$, we have an induced fiber metric $\langle \cdot, \cdot \rangle$ from $\gr_{x,-1}$.  Define the corresponding Riemannian metric $\langle \cdot , \cdot \rangle_I$ where $\langle \cdot , \cdot \rangle_I = \langle I \cdot , I \cdot \rangle $. 
Relative to $I$ and the Lie brackets on $\Lbra \cdot , \cdot \Rbra$ define \emph{the minimal torsion} of~$I$ by
$$T_0(v,w) = - I^{-1} \Lbra I(v), I(w) \Rbra.$$
We say that an affine connection $\nabla$ is \emph{strongly compatible with} $(M, E,g,I)$ if it is compatible with $(E,g)$, each subbundle $TM_{-j} = I^{-1} \gr_{-j}$ is parallel, and $\nabla T_0 =0$. Recall that $\nabla$ is compatible with a sub-Riemannian structure $(E,g)$ if it preserves orthonormal frames of $E$ under parallel transport.

The pairs $(I,\nabla)$ of $E$-gradings with strongly compatible connections are in one-to-one correspondence to adapted Cartan connections on the nonholonomic frame bundle $\psi: T\scrF \to \mathfrak{g}$.
The correspondence is as follows, Write $\psi = (\theta, \omega)$ with~$\theta$ and~$\omega$ taking values in $\mathfrak{g}_-$ and $\mathfrak{g}_0$ respectively. The form $\omega$ is a principal connection on~$\scrF$, inducing a linear connection $\nabla^{\gr}$ on $\gr_-$. If $h_u:T_{\pi(u)} M \to \ker \omega_u$, $u \in \scrF$ denotes the horizontal lift with respect to $\omega$, then
$$I(v) = u(\theta(h_u v)), \qquad v \in T_{\pi(u)} M.$$
We finally have $\nabla = I^{-1}\nabla^{\gr}I$.

The minimal torsion $T_0$ will then be the degree zero component of the torsion $T$ of $\nabla$ with respect to the $E$-grading $TM_{-1} \oplus \cdots \oplus TM_{-\stp}$. As this torsion corresponds to the Lie brackets in $\gr_-$, it cannot vanish. Let $\Gr_0 \subseteq \GL(\gr_-)$ be the principal bundle of fiberwise isometries of $\gr_-$ with its Lie algebra $\gr_0\subseteq \gr_-^* \otimes \gr_-$. We can then obtain corresponding fiberwise isometries and infinitesimal isometries on the tangent space by respectively $\Iso = \Iso(M, E, g, I) = I^{-1} \Gr_0 I$ and $\iso = I^{-1} \gr_0 I$. Observe that while these spaces depend on $I$, their restrictions to $E$ do not. Hence, elements $S \in \Iso$ and $s \in \iso$ are determined by respective properties
\begin{align*}
S|_E &\in \Gr_0|_E, & ST_0(v,w) &= T_0(Sv, Sw), \\
s|_E &\in \gr_0|_E, & sT_0(v,w) &= T_0(sv, w) + T_0(v, sw),
\end{align*}
for any $v, w \in TM$. That $\nabla$ is strongly compatible with respect to $(M,E,g,I)$ is equivalent to the holonomy group satisfying $\Hol_x(\nabla) \in \Iso$ for any $x \in M$.

We observe that the curvature of $\nabla$, $R(v,w) \in \iso_x$, $v,w \in T_xM$, corresponds to $\pr_0 \kappa$. On the other hand $\pr_{-} \kappa$ corresponds to $T- T_0$. More precicely, we have
$$K(h_u v_1, h_u v_2) = (u^{-1} I(T-T_0)(v_1, v_2) , u^{-1}IR(v_1, v_2) I^{-1} u) \in \frakg_- \oplus \frakg_0.$$

\section{Partial connections and canonical extension} \label{sec:ExtendPartial}
\subsection{Partial connection and canonical extensions} \label{sec:PartialConnection}
\emph{A partial (Cartan) connection} $\psi_{E} = ( \theta_{E}, \omega_E)$ over $E$ is a map $\psi_{E}: \pi^{-1}_* E = \hat E^{-1} \to  \mathfrak{g}_{-1} \oplus \frakg_0$ that is an isomorphism on any fiber, equivariant $\psi_{E}|_{u \cdot g_0} = \Ad(g_0^{-1}) \psi_{-1}|_u$, $g_0 \in G_0$, and satisfies
$$\psi_{E}( \tfrac{d}{dt} |_{t=0} u \cdot e^{ts}) = s, \qquad s \in \mathfrak{g}_0.$$
We will show here that there is a canonical way of extending such connection to a full Cartan connection on $\scrF$. Analogous to full Cartan connections, we say that a partial connection $\psi_{E} = ( \theta_E ,\omega_E)$ is \emph{adapted} if
$$\theta(w) = u^{-1} \pi_* w, \qquad w \in \hat E_{u}^{-1}.$$
For two adapted partial connections $\psi$ and $\tilde \psi$, remark that $(\psi - \tilde \psi)(\mathfrak{g}_{-1}) \subseteq \mathfrak{g}_0$.
\begin{lemma} \label{lemma:CartanExtension}
For any partial connection $\psi_{E}$ of $\scrF$, there exists a unique Cartan connection $\psi$ with curvature $\kappa$ on $\scrF$ such that
$$\psi |_{\pi_*^{-1} E} = \psi_{E}, \qquad \partial_b^{-1} \kappa =0.$$
Furthermore, if $\psi_E$ is adapted, then the corresponding full connection $\psi$ is adapted as well.
\end{lemma}
Observe that we could have also written the curvature condition as $\partial^*_b \kappa =0$.
\begin{proof}
We will build $\psi$ from $\psi_E$ step by step as follows. For any $A \in \frakg_{-1}$, define $H_A$ as the unique vector field on $\scrF$ satisfying $\psi(H_A) = \psi_E(H_A) = A$. Recall that
$$\partial_b(A \wedge \beta) = A \wedge \partial_b \beta, \qquad A \in \frakC^0 = \frakg,\ \beta \in \wedge \frakg_-^*.$$
Our first aim is to choose the correct image of $\psi^{-1}$ applied to $\frakg_{-2}$.
\begin{enumerate}[$\bullet$]
\item Observation~1: Consider the dual $(\partial_b^{-1})^\vee: \wedge \frakg_- \to \wedge \frakg_-$ to $\partial_b^{-1} :  \wedge \frakg_{-}^* \to \wedge \frakg_{-}^*$. Since these mappings have degree zero, the image of $\frakg_{-2}$ under $(\partial_b^{-1})^\vee$ is $(\wedge \frakg_{-})_{-2} = \wedge^2 \frakg_{-1}$.
\item Observation~2: The mapping $\psi_E^{-1}: \frakg_{-1} \to \Gamma(T\scrF)$, $A \mapsto H_A$ extends to a map $\wedge^2 \psi_E: \wedge^2 \frakg_{-1} \to \wedge^2 \Gamma(T\scrF)$.
\end{enumerate}
We combine the observations to get the following sequence of linear mappings of vector spaces
\begin{equation} \label{MappingtoVF}
\frakg_{-2}\overset{-(\partial_b^{-1})^\vee}{\to}\wedge^2\frakg_{-1}\overset{\wedge^2 \psi_E^{-1}}{\to} \wedge^2 \Gamma(T\scrF) \overset{\text{Lie brackets}}{\to} \Gamma(T\scrF).\end{equation}
Explicitly, for any $B \in \frakg_{-2} = (\wedge^1 \frakg_-)_{-2}$ if
$$(\partial_b^{-1})^\vee(-B) = \sum_{i=1}^r A_{1,i} \wedge A_{2,i} = (\wedge^2 \frakg_-)_{-2} = \wedge^2 \frakg_{-1}
,$$ 
then the image of \eqref{MappingtoVF} is given by
\begin{equation} \label{HB_extend} H_B = \sum_{i=1}^r [H_{A_{1,i}}, H_{A_{2,i}}],\end{equation}
which, by our construction, is uniquely defined by the element $(\partial_b^{-1})^\vee(-B)$ independent on how we choose to represent it. Furthermore, we must have $B = \sum_{i=1}^r [A_{1,i} ,A_{2,i}]$ by the definition of $(\partial_b^{-1})^\vee$.

Observe that with this definition, $\hat E^{-2} = \hat E^{-1} \oplus \{ H_B \, : \, B \in \frakg_{-2}\}$. Extend $\psi$ to $\hat E^{-2}$ by defining $\psi(H_B) = B$ for $B \in \frakg_{-2}$. We now observe that for the (still partially defined) curvature $\kappa$ of $\psi$
\begin{equation} \label{partialinversekappa}\partial_b^{-1} \kappa(B) = \sum_{i=1}^r\kappa(A_{1,i}, A_{2,i}) = -\sum_{i=1}^r \psi([H_{A_{1,i}}, H_{A_{2,i}}]) + B=0.\end{equation}
Repeating this argument for each step, that is, if $B \in \frakg_{-j-1}$, then $(\partial_b^{-1})^\vee(-B) \in (\wedge^2 \frakg_{-})_{-j-1} \subseteq \wedge^2 (\frakg_{-1} \oplus\cdots \oplus \frakg_{-j})$ and we can use \eqref{HB_extend} to define $H_B$ working iteratively for $j=1,2, \dots, \stp-1$. In this way, obtain a connection $\psi$ from the partial connection $\psi_E$ and the requirement $\partial_b^{-1} \kappa =0$, by virtue of the equivalence of \eqref{HB_extend} and \eqref{partialinversekappa}.

Finally, if $\psi_E$ is adapted, using the fact that $E_x =\gr_{x,-1}$ generates $(\gr_x, \Lbra \cdot , \cdot \Rbra)$ at every point $x \in M$ and that every $u \in \scrF_{x}$ is a Lie algebra isomorphism, it follows that $\psi$ must be adapted as well.
\end{proof}

\begin{remark}
Let $\psi = (\id + \alpha) \tilde \psi$ be two Cartan connections with respective curvatures $\tilde \kappa$ and $\kappa$. Assume that $\partial_b^{-1}\kappa =0$.
For a given $j \geq 1$ and $A \in \mathfrak{g}_{-j-1}$, let $B = \sum_{r=1}^m B_{r,1} \wedge B_{r,2} \in \wedge^2 \mathfrak{g}_{-j}$ be such that $\partial^{-1}_b\beta(A) =-\beta(B)$ for any $\beta \in \frakC^2$. Writing $(\id + \alpha)^{-1} = \id + \alpha_{inv}$ again, then by \eqref{secondKappa},
\begin{align} \label{thirdKappa}
\alpha(A) & = \sum_{r=1}^m (\id + \alpha) (d^{\tilde\psi}\alpha + \tilde \kappa)(B_{r,1} + \alpha_{inv}(B_{r,1}) , B_{r,2} + \alpha_{inv}(B_{r,2})) \\ \nonumber
& \quad + \sum_{r=1}^m (\id+\alpha) \left( - \alpha_{inv}(B_{r,1}) B_{r,2}   +\alpha_{inv}(B_{r,2})B_{r,1} \right. \\ \nonumber
& \qquad \qquad \qquad \qquad \qquad \qquad \qquad \left. - [ \alpha_{inv}(B_{r,1}) ,  \alpha_{inv}(B_{r,2})]\right) 
\end{align}
which iteratively determines $\alpha|_{\frakg_{\leq -2}}$ from $\tilde \kappa$ and $\alpha|_{\frakg_{-1}}$.
\end{remark}

\subsection{Unique extension for affine connections}
\emph{A partial (affine) connection} $\nabla^E$ on $E$ is a map $\Gamma(E) \times \Gamma(E) \to \Gamma(E)$, $(X, Y) \to \nabla_X^E Y$ which is tensorial in the first argument, linear in the second and satisfies Leibniz rule $\nabla_X^E fY = Xf Y + f\nabla^E_X Y$. From any adapted partial Cartan connection $\psi_E = (\theta_E, \omega_E)$ on~$\scrF$, we can define a partial affine connection by letting parallel transport of a frame correspond to curves whose derivatives are in $\ker \omega_E$. As we define it from a partial Cartan connection in the $G_0$-principal bundle $\scrF$, the parallel transport along any loop tangent to $E$ from $\nabla^E$ must give us an element in $\Gr_0|_E$. In other words, adapted partial Cartan connections are in one-to-one correspondence with partial affine connection on $E$ whose horizontal holonomy group $\Hol_{E,x}(\nabla^E)$ satisfies $\Hol_{E,x}(\nabla^E) \subseteq \Gr_0|_E$. See \cite{CGJK19} for the definition of the horizontal holonomy group.

In order to write the corresponding extension condition for $\partial_b^{-1} \kappa =0$ for affine connections and gradings, we introduce the mapping $\chi: TM \to \wedge^2 TM$ as the pseudo-inverse $\chi = -T_0^{-1}$ of $T_0$. Then we can rewrite Lemma~\ref{lemma:CartanExtension} as follows.
\begin{lemma}
If $\nabla^E$ is a partial affine connection with horizontal holonomy group in $\Gr_0|_E$, then there exists a unique $E$-grading $I$ and strongly compatible connection $\nabla$ such that the curvature $R$ and the torsion $T$ of $\nabla$ satisfy
\begin{equation} \label{Extend} R(\chi(\cdot)) =0, \qquad (T-T_0)(\chi(\cdot)) = T(\chi(\cdot)) + \pr_{\leq -2} =0,\end{equation}
where $\pr_{\leq -2}: TM \to TM_{-2} \oplus \cdots \oplus TM_{-\stp}$ is the projection to degrees lower than $-1$ with kernel $E$.
\end{lemma}
We look explicitly at how the grading can be constructed. Let $E = TM_{-1}$ and define $\nabla_X Y = \nabla^E_{X}Y$. Following an induction argument, we assume that $E^{-j} = TM_{-1} \oplus \cdots \oplus TM_{-j}$ has been defined with inner products and that $\nabla_X Y$ has been defined whenever $X$ and $Y$ takes values in $E^{-j}$. Furthermore, we assume that $\chi: TM_{\geq -j} \to (\wedge^2 TM)_{\geq -j}$, $T_0, T: (\wedge^2 TM)_{\geq -j} \to TM_{\geq-j}$ and $R: (\wedge^2 TM)_{\geq -j} \to (TM^*\otimes TM)_{\geq -j}$ has been defined with
$(T-T_0)(\chi(TM_{\geq-j})) =0$ and $R(\chi(TM_{\geq -j})) =0$. These conditions are trivially satisfied for $j=-1$. We continue to the next step as follows.
\begin{enumerate}[\rm (i)]
\item Consider the map
$$(\wedge^2 TM)_{\geq -j-1}  \to TM_{\geq -j-1} \bmod TM_{\geq -j}, \quad X \wedge Y \mapsto [X,Y] \bmod E^{-j}$$
and let $\calX_j \subseteq (\wedge^2 TM)_{\geq -j-1}$ denote the orthogonal complement of the kernel of this map in $(\wedge^2 TM)_{\geq -j-1}$.
\item Define
\begin{align*}
T&M_{-j-1} := \left\{\sum_{i=1}^r T_0(X_{1,i},X_{2,i})  \, : \, \sum_{i=1}^r X_{1,i} \wedge X_{2,i} \in \Gamma(\calX_j) \right\} \\
&= \left\{ \sum_{i=1}^r (\nabla_{X_{1,i}} X_{2,i} - \nabla_{X_{2,i}} X_{1,i} - [X_{1,i},X_{2,i}]) \, : \, \sum_{i=1}^r X_{1,i} \wedge X_{2,i} \in \Gamma(\calX_j) \right\}.
\end{align*}
Use $T_0$ from $\calX_j$ to $TM_{-j-1}$ to define an inner product.
\item Since we can now write any section $Y$ in $TM_{-j-1}$ as $Y= \sum_{i=1}^r T_0(X_{1,i}, X_{2,i})$, we can define $$\nabla_{X} Y := \sum_{i=1}^r T_0(\nabla_X X_{1,i}, X_{2,i}) + \sum_{i=1}^r T_0(X_{1,i}, \nabla_X X_{2,i})$$
and
$$\nabla_{Y} := \sum_{i=1}^r \left(  \nabla^2_{X_{i,1}, X_{i,2}} - \nabla^2_{X_{2,1}, X_{i,1}} \right).$$
This definition ensures that $R(\chi(E^{-j-1})) =0$ and $(T-T_0)(\chi(E^{-j-1})) =0$.
\end{enumerate}
Proceeding this way iteratively, we define the grading and connection.

\begin{remark}
To produce a partial connection $\nabla^E$ with the correct horizontal holonomy group, we can proceed in the following way. Let $(M, E, g)$ be a manifold with constant symbol $\frakg_{-}$. Let $U$ be a local section of the nonholonomic frame bundle~$\scrF$. Let $I:TM \to \gr_-$ be an $E$-grading and define locally
$$X_A(x) = I(x)^{-1} U(x)^{-1}A, \quad A\in\frakg_-.$$
In other words, we have defined a subspace of vector fields $\{ X_{A} \, : \, A \in \frakg_{-}\}$ such that their Lie brackets coincide with those in $\frakg_-$ modulo terms of higher orders. In other words, if $A \in \frakg_{-i}$, $B \in \frakg_{-j}$, then for any $x \in M$
$$[X_{A}, X_{B}] |_x = X_{[A,B]}|_x \mod \spn\{X_C|_x \, : \, C \in \frakg_{-i-j+1} \oplus \cdots \oplus \frakg_{-1} \}.$$
Now, consider the flat connection $\tilde \nabla$ given by $\tilde \nabla X_A =0$ for any $A \in \frakg_{-}$ 
and write $\tilde \nabla^E = \nabla_{|E}|_E$. Then any other partial connection $\nabla^E$ with $\Hol_{E,x}(\nabla^E) \subseteq \Gr_0|_E$ can be written as
$$\nabla^E = \tilde \nabla^E + \eta, \qquad \eta \in \Gamma(T^*M \otimes \gr_0|_E).$$
Then such locally defined connections over the local sections can be glued together, as usual.
\end{remark}

\subsection{Holonomy and horizontal holonomy of a Cartan connection}
The map $\chi$ is an example of what is called \emph{a selector} in \cite{CGJK19}, where it is also shown that the condition $R(\chi(\cdot)) =0$ in \eqref{Extend} ensures that the group formed by parallel transport relative to $\nabla$ along all loops equals that of parallel transport along just loops tangent to $E$. We want to show a similar result for Cartan connections.

Let $\psi: T\scrF \to \frakg$ be a Cartan connection on the nonholonomic frame bundle $G_0 \to \scrF \to M$. Let $G$ be the simply connected Lie group corresponding $\frakg$. Define a $G$-principal bundle $G \to \hat \scrF \to M$ by $\hat \scrF = \scrF \times_{G_0} G$, i.e., the product $\scrF \times G$ quotiented out by the equivalence relation $(u\cdot g_0, g_0^{-1}g) \sim (u,g)$, $u\in \scrF$, $g_0 \in G$, $g\in G$. Write $ug$ for the equivalence class of $(u,g)$. Let $\inc:\scrF \to \hat \scrF$ be the inclusion given by $u \mapsto u1$. We will follow \cite{CapGrover14}.

Given $\psi$, we define the principal connection $\hat \psi : \hat \scrF \to \frakg$ uniquely by the condition $\inc^* \hat \psi=\psi$. In other words, for any $A\in \frakg$, write $H_A$ for the vector field of $T\scrF$ satisfying $\psi(H_A) = A$, and write $\xi_A$ for the vector field on $\hat \scrF$ defined by
$$\xi_A(ug) = \frac{d}{dt} uge^{tA}|_{t=0}.$$
From the condition $\inc^* \hat \psi = \psi$, if we define a vector field on $\hat \scrF$ such that
$$\hat H_A(ug) = (\inc_* H_{\Ad(g) A}(u)) \cdot g,$$
then $\hat \psi(\hat H_A) = A$. Remark that the vector field $\hat H_A$ is well defined since $H_A(u) \cdot g_0 = H_{\Ad(g_0)A}(u\cdot g_0)$. Furthermore, for any $s \in \frakg_0$, we have
\begin{align*}
    \hat H_s(ug) & = (\inc_* H_{\Ad(g) s}(u)) \cdot g = \frac{d}{dt} (u \cdot e^{t\Ad(g)(s)})g |_{t=0} \\
    & = \frac{d}{dt} uge^{ts} |_{t=0} = \xi_s(ug),
\end{align*}
and
$$\ker \hat \psi = \spn \{ \hat H_A - \xi_A \, : \, A \in \frakg_-\}.$$
Just as with any principal connection on a principal bundle, we can define holonomy group $\Hol_{ug}(\hat \psi) \subset G$ at any $ug \in \hat \scrF$. These are all conjugate, and we define the holonomy $\Hol(\psi)$ of $\psi$ as the conjugacy class of these groups. Similarly, we define the horizontal holonomy $\Hol_E(\psi)$ as the conjugacy class obtained by only considering loops tangent to $E$. In other words, for any $ug \in \hat \scrF_x$, $x \in M$, define
$$\Hol_{E,ug}(\hat \psi) = \left\{ \tilde g \in G \, : \, \begin{array}{c}
\text{There exists $\gamma:[0,1] \to \hat \scrF$, $\gamma(0) = ug$}\\
\text{$\dot \gamma(t) \in \spn\{\hat H_A - \xi_A \, :A\in \frakg_{-1} \}$, $\gamma(1) = ug\tilde g \in \hat \scrF_x$}
\end{array}\right\}.$$
Remark that from this definition, $\Hol_E(\psi)$ only depends on the partial connection $\psi_E = \psi|_{(\pi_*)^{-1}E}$, which means that we can define $\Hol_E(\psi) = \Hol(\psi_E)$

We then have the following result.
\begin{theorem} \label{th:HolE}
Let $\psi$ be a Cartan connection on $\scrF$ with curvature $\kappa$ such that $\partial_b^{-1} \kappa =0$. Then $\Hol(\psi_E) = \Hol(\psi)$.
\end{theorem}
Remark the result here can never give an ``if and only if''-statement. For example, if a partial connection $\psi_E$ is such that $\Hol(\psi_E) = G$, then any extension of $\psi_E$ to a full Cartan connection will have the same holonomy.
\begin{proof}
By definition, for $A,B \in \frakg$,  the Lie brackets of the corresponding vector fields on $\hat \scrF$ are given by
$$[\xi_A, \xi_B] = \xi_{[A,B]}, \qquad [\xi_A, \hat H_B] = \hat H_{[A,B]}.$$
Furthermore, if $A, B \in \frakg_-$, then
$$[\hat H_A, \hat H_B] = \hat H_{[A,B] -\hat \kappa(A,B)},$$
where $\hat \kappa(ug)(A,B) = \Ad(g^{-1}) \kappa(u)(\Ad(g) A, \Ad(g) B)$. 
It follows that for $A,B \in \frakg_-$,
$$[\hat H_A - \xi_A, \hat H_B - \xi_B] = -(\hat H_{[A,B]} - \xi_{[A,B]}) - \hat H_{\hat \kappa(A,B)}.$$

Define subbundles $\calH_E$ and $\calH$ of $T\hat \scrF$ by
$$\calH_E = \spn \{ H_A - \xi_A \, : \, A\in \frakg_{-1} \}, \qquad \calH = \{H_A -\xi_A \, : \, A \in \frakg_- \}.$$
Let $\scrO_{E,ug}$ and $\scrO_{ug}$ be the respective orbits of these subbundles at $ug \in \hat \scrF$. In other words, $\scrO_{E,ug}$ consist of all points in $\hat \scrF$ such that can be reached by curve tangent to $\calH_E$, with similar definition for $\scrO_{ug}$. These orbits determine the holonomy since
$$\Hol_{E,ug}(\hat \psi) = \{ \tilde g \in G \, : \, ug\tilde g \in \scrO_{E,ug} \},$$
and similarly for $\Hol_{ug}(\hat \psi)$. Since $\calH_E \subseteq \calH$, it follows that $\scrO_{E,ug} \subseteq \scrO_{ug}$. To show that they are indeed equal, write $\Lie_{ug} \calH_E$ and $\Lie_{ug} \calH$ for the subspace of $T_{ug} \hat \scrF$ spanned by vector fields in respectively $\calH_E$ and $\calH$ and their iterated Lie brackets. By \cite[Lemma~2.12]{CGJK19}, if we show that $\Lie_{ug} \calH_E = \Lie_{ug} \calH$ for every $ug \in \hat \scrF$, then we will also have $\scrO_{E,ug} =\scrO_{ug}$ for every $ug \in \scrF$. In order to show this result, it is sufficient to show $\hat H_B(ug) - \xi_{B}(ug) \in \Lie_{uq} \calH_E$ for any $B \in \frakg_-$.

As in the proof of Lemma~\ref{lemma:CartanExtension}, define $(\partial_b^{-1})^\vee: \frakg_- \to \wedge^2 \frakg_-$ as the dual of $\partial_b^{-1}: \wedge^2 \frakg_-^* \to \frakg_-^*$. If $\partial_b^{-1}\kappa =0$, if follows that if $(\partial_b^{-1})^\vee(B) = \sum_{i=1}^r A_{i,1} \wedge A_{i,2}$ for $B\in g_{-j}$, $j \geq 2$, then
$$\hat H_B - \xi_B = - \sum_{i=1}^r[\hat H_{A_{i,1}} - \xi_{A_{i,1}}, \hat H_{A_{i,2}} - \xi_{A_{i,2}}],$$
where $A_{i,1}, A_{i,2} \in \frakg_{-j+1} \oplus \cdots \oplus \frakg_{-1}$. Using this argument iteratively, it follows that any $\hat H_B - \xi_{B}$ can be obtained by iterated brackets of vector fields of the form $\{ H_A- \xi_A \, : \, A \in \frakg_{-1}\}$. Hence, $\hat H_B(ug) - \xi_{B}(ug) \in \Lie_{uq} \calH_E$ for any $B \in \frakg_-$ and $ug \in \hat \scrF$, which completes the proof.
\end{proof}

\section{Finding normalization conditions} \label{sec:Normalization}

\subsection{Canonical connection}
We are now ready to prove the main result.
\begin{proof}[Proof of Theorem~\ref{th:main}]
Let again $\psi = (\id+\alpha) \tilde \psi$, with both connections adapted, and both with curvatures satisfying $\partial_b^{-1} \kappa = \partial_b^{-1} \tilde \kappa = 0$. Let $\psi_E$ and $\tilde \psi_E$ be the partial connections generating $\psi$ and $\tilde \psi$ respectively. If we consider the equation \eqref{KappaDegree} for homogeneity $1$, then we have
$$\kappa_1 = \partial \alpha_1 +  \tilde \kappa_1,$$
and applying $\partial_b^{-1}$ to this equation, it follows that $\partial_b^{-1} \partial \alpha_1 =0$ and $\alpha_1 \in \frakP^1$.
Observe that by the Bianchi identity, $\partial \kappa_1 = \partial \tilde \kappa_1 =0$, so $\kappa_1, \tilde \kappa_1 \in \frakP^2$.

We observe that $\Pi \kappa_1$ uniquely determines $\kappa_1$ and that
$$\Pi(\kappa_1 - \tilde \kappa_1) \in \Pi \partial P^\infty \frakC_1^1 = \Pi P^\infty \partial \frakC_1^1 = \Pi \partial \frakC_1^1.$$
Require now that $\langle \Pi  \kappa_1 , \Pi \partial \frakC^1_1 \rangle =0$, which will uniquely determine $\Pi  \kappa_1$. The partial connection $\psi_E = (\id + \alpha_E) \tilde \psi_E$ will be determined by condition $\alpha_1|_E = \alpha_E$.

Finally, we rewrite the orthogonality condition. Since $\partial$ and $P^{\infty}$ are operators of homogeneity degree  zero, we can write the condition as  $\langle \Pi \kappa_1 , \Pi \partial \frakC^1 \rangle =0$. Using the properties of $\Pi$, including the fact that it is an orthogonal projection
\begin{align*}
0 & = \langle \Pi \kappa_1 , \Pi \partial \frakC^1 \rangle  = \langle \Pi \kappa_1 ,  \partial \frakC^1 \rangle =  \langle \partial^* (\id - \partial_b^{-1} \partial_b)  \kappa_1 ,  \frakC^1 \rangle.
\end{align*} 
The result follows.
\end{proof}

\begin{corollary}
On the nonholonomic frame bundle, there exists a unique Cartan connection $\psi$ with curvature $\kappa$ such that,
$$\partial^{-1}_b \kappa =0, \qquad \partial^* (\iota_A \kappa_1) =\partial^* (\iota_A  \partial_b^{-1} \partial_b \kappa_1), \qquad \text{for any $A \in \mathfrak{g}_{-}$}.$$
\end{corollary}

\begin{proof}
Observe first that since $\partial_b^{-1} \kappa_1 =0$, we have $\partial_b^{-1} (\id - \partial_b^{-1} \partial_b) \kappa_1 =0$ and then also $\partial^* (\id - \partial_b^{-1} \partial_b) \kappa_1 =0$. If we consider a general element, $B \wedge A^* \in \frakC^1$, with $A \in \frakg_{-}$, $B \in \frakg$, then
\begin{align*}
0 & = \langle \partial^* (\id - \partial_b^{-1} \partial_b) \kappa_1,B \wedge A^* \rangle = \langle  (\id - \partial_b^{-1} \partial_b) \kappa_1, \partial(B \wedge A^* ) \rangle \\
&  = \langle  (\id - \partial_b^{-1} \partial_b) \kappa_1, \partial B \wedge A^* + \partial_b (B \wedge A^*) \rangle  \\
&  = -  \langle \iota_A (\id - \partial_b^{-1} \partial_b) \kappa_1,  \partial B \rangle  = -  \langle \partial^* (\iota_A (\id - \partial_b^{-1} \partial_b) \kappa_1), B \rangle .
\end{align*}
The result follows.
\end{proof}

\begin{remark} \label{re:Complements}
When computing $\psi$, we can require that $\kappa$ is orthogonal to $\partial_b \frakC^1_{\geq 1}$ and that $\kappa_1$ is $\partial_b^*$ orthogonal to $\Pi \partial \frakC_1^1$. Furthermore, since  $\Pi \partial \frakC_1^1 = (\id - \partial_b^{-1} \partial_b) \partial \frakC_1^1 \mod \partial_b \frakC^1_1$, we can use orthogonality to $(\id - \partial_b^{-1} \partial_b) \partial \frakC_1^1$ instead. If $\psi = (\id  + \alpha)\tilde \psi$ with $\tilde \psi$ a known connection, then we first solve the equations
$$\partial^* (\id - \partial^{-1}_b \partial_b) (\partial \alpha_1 + \tilde \kappa_1) =0, \qquad \partial_b^{-1} (\partial \alpha_1+ \tilde \kappa_1) =0.$$
\end{remark}

\subsection{Cartan connections on the nonholonomic frame bundle}
We can define our condition in terms of $E$-gradings $I$ and  affine connections $\nabla$, with the torsions $T$ and curvatures $R$. For two-forms $\Omega^2(M)$ on $M$, let $\Omega^2_{\Jac}(M)$ denote the subspace of forms that satisfy the property
$$\alpha(T_0(v_1, v_2), v_3) + \alpha(T_0(v_3, v_1), v_2) + \alpha(T_0(v_2, v_3), v_1) =0,$$
with $\alpha \in \Omega^2(M), v_1, v_2, v_3 \in TM$.
For any $\alpha \in \Omega^2$, we write $\alpha_{\Jac}$ for its orthogonal projection to $\Omega_{\Jac}^2$. Such a projection induces a projection also for forms with values in a vector bundle, for which we will use the same symbol.

\begin{theorem} Let $(M,E,g)$ be a sub-Riemannian manifold of constant symbol.
There is a unique $E$-grading $I$ and strongly compatible connection $\nabla$ satisfying the extension condition
$$R(\chi(\cdot)) =0, \qquad (T-T_0)(\chi(\cdot))  =0,$$
and
\begin{align*}
\langle T_{\Jac}(v, \cdot), s \rangle & =0, \qquad &  v \in E, s \in \iso, \\
\langle T_{\Jac}(v, \cdot ), T_0(w, \cdot) \rangle & =0, \qquad &v \in TM_{-j-1}, w\in TM_{-j}.
\end{align*}
\end{theorem}
\begin{proof}
The result follows from the identity
$$\partial_b^* (\wedge^3 \frakg_-^*) = \{ \circlearrowright [C_1,C_2]^* \wedge C_3^* \, : \, C_1, C_2, C_3 \in \frakg_{-}\},$$
meaning that $(\id - \partial_b^{-1} \partial_b)\kappa_1$ correspond to $(T_{\Jac})_1$.
\end{proof}

\begin{remark}
In comparison, the Morimoto connection is the solution of the following equations with $v \in TM_{-i}$, $w \in TM_{-j}$, $i<j$,
$$\langle R(\chi(v)), s \rangle = \langle T(v, \cdot), s \rangle, \qquad \langle T(\chi(v)), w \rangle =  - \langle T(v, \cdot) , T_0(w, \cdot) \rangle,$$
see \cite{Gro20}. These conditions are normalization conditions for both the grading and connections simultaneously. Observe that forcing the left hand sides of the equations to vanish is our extension condition, while our normalization condition is forcing the right hand sides to vanish for some values, when $T$ is replaced with $T_{\Jac}$.
\end{remark}

\section{Small dimensional examples} \label{sec:Small}
The simplest example is of course the case $\frakg_{-} = \frakg_{-1}$. Then $\partial_b = 0$, $P^\infty  =\id$, $\kappa = \kappa_1$ and we get the condition $\partial^* \kappa = 0$, which will give us the Levi-Civita connection. Going beyond this example, we look first at explicit computations in low dimensions.

\subsection{Smallest Heisenberg algebra} \label{sec:3dHeisenberg}
We consider the Heisenberg algebra $\frakg_{-} = \frakg_{-2} \oplus \frakg_{-1}$, where $\frakg_{-1} = \spn \{ A_1, A_2\}$, $\frakg_{-2} = \spn \{ B\}$ and $\frakg_0 = \spn \{s\}$, where
$$[A_1, A_2] = B, \qquad s(A_{1}) = A_{2}, \qquad s(A_{2}) = - A_{1}, \qquad s(B) =0.$$
We consider the inner product on $\frakg_{-1}$ such that $A_1, A_2$ form an orthonormal basis. All sub-Riemannian manifolds with growth vector $(2,3)$ have $\frakg_{-}$ as their constant symbols.

\subsubsection{Condition for the unique connection} We use Remark~\ref{re:Complements}.
We see that
$$\partial_b \frakC^1_{\geq 1} =  \left\{ A \wedge A_{1}^* \wedge A_{2}^* \, : \, A \in \frakg_{-1} \oplus \frakg_{0}  \right\} $$
while
$$\partial \frakC^{1}_1 = \spn\{ A \wedge A_1^* \wedge A_2^* , B \wedge B^* \wedge A^*  \, : \, A \in \mathfrak{g}_{-1} \} ,$$
so
$$\Pi \partial \frakC_1^1= \spn\{ B \wedge B^* \wedge A^*  \, : \, A \in \mathfrak{g}_{-1} \}. $$
It follows that there is a unique connection satisfying
$$\kappa(A_1, A_2) =0, \qquad \kappa_1 =0.$$
Viewed from the manifold, the condition is
$$R(v_1, v_2) =0, \qquad T = T_0 + T_2, \qquad v_1, v_2 \in E,$$
with respect to the grading described below.

\subsubsection{Description in a local basis}
Let $(M, E, g)$ be any sub-Riemannian manifold with growth vector $(2,3)$. Let $X_1, X_2$ be any choice of orthonormal basis for $E$. Define $JX_1 = X_2$ and $JX_2 = -X_1$. Any compatible affine partial connection on $E$ preserves the metric, and thus, it is uniquely determined by two functions $\alpha_1$ and $\alpha_2$ given by
$$\nabla_{X_i} X_j = \nabla^E_{X_i} X_j = \alpha_i J X_j.$$
We then set $TM_{-2} = \spn\{ Z\} = \spn \{ - T_0(X_1, X_2)\}$, with
$$Z = [X_1, X_2] -\nabla_{X_1} X_2 + \nabla_{X_2} X_1 = [X_1, X_2] + \alpha_1 X_1 +\alpha_2 X_2.$$
We extend $J$ by $JZ =0$. We have $-\nabla_{X_i} Z = \nabla_{X_i} T_0(X_1, X_2) = T_0(\nabla_{X_i} X_1, X_2) + T_0(X_1, \nabla_{X_i} X_2) =0$, and
$$\nabla_{Z} = \nabla_{X_1, X_2}^2 - \nabla_{X_1, X_2}^2 =  (X_1 \alpha_2 - X_2 \alpha_1) J.$$
This completes the extension conditions. For the final normalization, we see that
$$0 = \langle Z, T(X_1, Z) \rangle = Z^*[X_1, Z] = Z^* [X_1, [X_1,X_2]] + \alpha_2$$
and similarly, we must have $0 = \langle Z, T(X_2, Z) \rangle = Z^* [X_2, [X_1,X_2]] - \alpha_1$. In conclusion, if $[X_j, [X_1,X_2]] = f_j [X_1, X_2] \bmod E$, then the canonical connection and grading of $TM$ is given by
$$Z = [X_1, X_2] - f_2 X_1 + f_1 X_2.$$
$$\nabla = (f_2 X_1^* - f_1 X_2^* - (X_1 f_1 + X_2 f_2) Z^*) \otimes J.$$


\subsubsection{Global description} We can define uniquely, up to sign, a one-form $\beta$  with $\ker \beta =E$ and $\| d\beta|_{\wedge^2 E}\| = 1$. Define \emph{the Reeb vector field} $Z$ as the unique vector field such that $\beta(Z) = 1$ and $d\beta(Z, \cdot) =0$. Define $TM_{-2} = \spn \{ Z\}$, making $Z$ orthogonal to $E$ and a unit vector field with respect to $g_I$. If $\nabla^I$ is the Levi-Civita connection of $g_I$, then our canonical connection $\nabla$ is
$$\nabla Z =0, \qquad  \nabla_X Y = \pr_{-1} \nabla_X^I Y, \qquad \nabla_Z = \nabla_{X_1,X_2}^2 - \nabla_{X_2,X_1}^2, \qquad X,Y \in \Gamma(E).$$
Note that $T_1(Z, \cdot) = 0$ corresponds to the condition $d\beta(Z, \cdot) =0$, which justifies our choice of the grading.

\begin{remark}
To compare it with earlier suggested connections, if we define a connection $\nabla''$ for $X,Y \in \Gamma(E)$ by,
$$\nabla'' Z =0 , \qquad  \nabla''_X Y  = \pr_{-1} \nabla_X^I Y, \qquad \nabla_Z ''Y = [Z,Y] + \frac{1}{2}\sharp (\calL_Z g_I)(Y, \cdot),$$
then this is the connection defined by Hladky, \cite{hladky2009connections}.
As shown in \cite{Gro20}, the Morimoto connection $\nabla'$ is given for $\tilde X, \tilde Y\in \Gamma(TM)$,
$$\nabla'_{\tilde X} \tilde Y = \nabla''_{\tilde X} \tilde Y + \frac{1}{2} R''(\chi(\tilde X)) \tilde Y,$$
and with the same grading given by the Reeb vector field.
\end{remark}

\subsection{Rolling Carnot algebra}
We consider the Carnot algebra $\frakg_{-} = \frakg_{-3}  \oplus \frakg_{-2} \oplus \frakg_{-1}$, with $\mathfrak{g}_{-1} = \spn\{ A_1, A_2 \}$, $\frakg_{-2} = \spn \{ B\}$, $\frakg_{-3} = \spn \{ C_1, C_2\}$, with
$$[A_1, A_2] = B, \qquad [A_j, B] = C_j.$$
Then we also have $\frakg_0 = \spn \{ s\}$, satisfying
$$s(A_1) = A_2, \quad s(A_2) = -A_1, \quad s(B) =0, \quad s(C_1) = C_2, \quad s(C_2) = - C_1.$$

\subsubsection{Conditions for the curvature}
We see that
$$\partial_b \frakC_{\geq 1}^1 = \spn\{ W_1 \wedge A_1^* \wedge A_2^* , W_2 \wedge A^* \wedge B^* \, : \, A \in \frakg_{-1}, W_1 \in \frakg_{\geq -1}, W_2 \in \frakg_{\geq -2}   \}.$$
Furthermore,
$$(\id - \partial_b^{-1} \partial_b) \frakC_1^1 = \spn\{ C_j \wedge C^*_j \wedge A^*_j , \, : \, j=1,2\}.$$
Hence, the restrictions are
\begin{equation} \label{RollingCond} \kappa(A_1, A_2) = \kappa(A_1,B) =\kappa(A_2,B) =0,\end{equation}
and
$$\langle \kappa(C_1, A_1) , C_1 \rangle = \langle \kappa(C_2, A_2), C_2 \rangle =0.$$
On the manifold, this is formulated as
$$R(v, v_2) = R(v, w) =0,  \qquad T(v, v_2) = T(v, w) =0,$$
$$\langle T(v, T_0(v,w)), T_0(v,w)\rangle =0  \qquad v, v_2 \in E, w \in TM_{-2},$$
with the grading defined below.

\subsubsection{Solving on the frame bundle}
We again solve first for $\alpha_1$. Define
\begin{align*}
0 = & (\partial \alpha_1 + \tilde \kappa_1)(A_1, A_2) = [\alpha_1(A_1), A_2] + [A_1, \alpha_1(A_2)] - \alpha_1(B) + \tilde \kappa_1(A_1, A_2) \\
0 = & (\partial \alpha_1 + \tilde \kappa_1)(A_j, B)=  [A_j, \alpha_1(B)] - \alpha_1(C_j) + \tilde \kappa_1(A_j, B) \\
0 = & \langle C_j , (\partial \alpha_1 + \tilde \kappa_1)(A_j, C_j)=  \langle C_j , [A_j, \alpha_1(C_j)] + \tilde \kappa_1(A_j, C_j) \rangle 
\end{align*}
which has solutions
\begin{align*}
\alpha_1(C_j) & = - \langle C_j, \tilde \kappa_1(A_j, C_j) 
\rangle B,\\
\alpha_1(B) & = - \langle \alpha(C_2) - \tilde \kappa_1(A_2, C_j), B\rangle A_1 +\langle \alpha(C_1) - \tilde \kappa_1(A_1, C_j), B\rangle A_2, \\
\alpha_1(A_1) & = \langle A_1,   - \alpha_1(B) + \tilde \kappa_1(A_1, A_2) \rangle s, \\
\alpha_1(A_2) & = \langle A_2,  - \alpha_1(B) + \tilde \kappa_1(A_1, A_2) \rangle s.
\end{align*}
The rest of $\alpha$ can be found by \eqref{thirdKappa}.

\subsubsection{Description locally on the manifold}
Let $(M, E, g)$ be a sub-Riemannian manifold with growth vector $(2,3,5)$. Let $X_1$, $X_2$ be an arbitrary orthonormal basis of~$E$.
Any compatible affine partial connection on $E$ can be expressed by functions
$$\nabla_{X_i}^E X_j =\eta_i J X_j.$$
where $JX_1 = X_2$ and $JX_2 = -X_1$. Define
$$Y = -\nabla_{X_1}^E X_2 + \nabla_{X_2}^E X_1 + [X_1,X_2] = \eta_1 X_1 + \eta_2 X_2 + [X_1, X_2].$$
Let $TM_{-2} = \spn \{ Y\}$ and define $\nabla_{X_j} Y =0$ to make sure $T_0$ is parallel. Finally, define $JY = 0$ and
$$\nabla_Y = \nabla_{X_1, X_2}^2- \nabla_{X_2, X_1}^2 = (X_1 \eta_2 - X_2 \eta_1)J.$$
Next, we define
$$Z_j = [X_j, Y] - \nabla_{X_j} Y + \nabla_Y X_j = [X_j, Y] +(X_1 \eta_2 - X_2 \eta_1)JX_j.$$
Define $JZ_1 = Z_2$ and $JZ_2 = -Z_1$. In order for $T_0$ to be parallel, we must have
$$\nabla_{X_i} Z_j = \eta_i JZ_j, \qquad \nabla_{Y} Z_j = (X_1 \eta_2 - X_2 \eta_1)JZ_j,$$
and
$$\nabla_{Z_j} = \nabla_{X_j, Y}^2 -\nabla_{Y, X_j}^2 = (X_jX_1 \eta_2 - X_j X_2 \eta_1 - Y\eta_j)J.$$

The normalization condition is given by
\begin{align*}
\langle Z_j, T(Z_j, X_j) \rangle & = \langle Z_j, [X_j, Z_j] \rangle =0.
\end{align*}
Define $[X_j,[X_j, [X_1, X_2]]] = f_{1,j} [X_1,[X_1, X_2]] + f_{2,j} [X_2, [X_1, X_2]] \bmod E^{-2}$. We then observe that
\begin{align*}
{[X_1, Z_1]} & = [X_1, [X_1, Y]] \bmod E^{-2} \\
& =  [X_1, [X_1, [X_1, X_2]]] + [X_1,[X_1, \eta_1 X_1 + \eta_2 X_2]]\bmod E^{-2} \\
& = f_{1,1} Z_1 + f_{2,1} Z_2 + \eta_2 Z_1 \bmod E^{-2}.
\end{align*}
Similarly, $[X_2, Z_2] =  f_{1,2} Z_1 + f_{2,2} Z_2 - \eta_1 Z_2.$ Hence, we have
$$f_{2,2} =\eta_1, \qquad f_{1,1} = -\eta_2,$$
which determines the connection.

\section{Contact manifolds} \label{sec:Contact}
\subsection{Heisenberg Carnot algebras}
Let $\mathfrak{g}_{-1}$ be a $2n$-dimensional vector space with a symplectic form $\varsigma$ and let $\mathfrak{g}_{-2} = \spn\{B\}$ be a one-dimensional vector
space with basis $B$. We introduce a Lie algebra structure on $\mathfrak{g}_- = \mathfrak{g}_{-2} \oplus \mathfrak{g}_{-1}$ by making $\mathfrak{g}_{-2}$ the center and defining
$$[A_1,A_2] = \varsigma(A_1,A_2) B, \qquad A_1, A_2 \in \frakg_-.$$
The algebra $\mathfrak{g}_-$ is called the \emph{Heisenberg algebra}. Let $\langle \cdot, \cdot \rangle$ be an inner product on~$\mathfrak{g}_{-1}$. Define a complex structure $J$ and symmetric map $\Lambda$ such that
$$\varsigma(v,w) = \langle v, \Lambda J w \rangle =  \langle v, J\Lambda w \rangle.$$
By rescaling $\varsigma$ and $B$ correspondingly, we may assume that the maximal eigenvalue of $\Lambda$ is $1$. The special case $\Lambda = \id$ is called \emph{the standard Heisenberg algebra}.

Since we completed the case $n=1$ in Section~\ref{sec:3dHeisenberg}, we will only consider $n \geq 2$.
Let $1 = \lambda[1] > \lambda[2] > \cdots > \lambda[l]$ be the eigenvalues of $\Lambda$, listed without repetition. Assume that the eigenspace of $\lambda[j]$ is of dimension $2 k[j]$. Then we can identify $\mathfrak{g}_-$ with $\mathbb{C}^n \times \mathbb{R}$, $z = (z[1], \dots, z[l]) \in \mathfrak{g}[1] \oplus \cdots \oplus \mathfrak{g}[l] =\mathbb{C}^{k[1]} \times \cdots \times \mathbb{C}^{k[l]}$, such that
$$[z_1 + c_1B, z_2 +c_2B] = \sum_{j=1}^l \lambda[j] \, \mathrm{Re} \langle iz_1[j],  z_2[j] \rangle_{\mathbb{C}^{k_j}} B.$$
If $\blambda =(\lambda[1], \lambda[2], \cdots ,\lambda[l])$, we write the corresponding Carnot algebra $\mathfrak{g}_- = \mathfrak{g}_-^{\blambda}$. In order for a linear map $S: \mathfrak{g}_- \to \mathfrak{g}_-$ to be an isometry, its restriction to $\mathfrak{g}_{-1}$ needs to preserve any subcomponent $\mathbb{C}^{k[j]}$ as well as being a unitary map of this space. We can then identify $S \in G_0$ with $(S[1], \dots, S[l]) \in U(k[1]) \times \cdots \times U(k[l])$, which allows us to identify $\mathfrak{g}_0$ with $\mathfrak{u}(k[1]) \times \cdots \times \mathfrak{u}(k[l])$.

\subsection{Contact manifolds with constant symbol}
We consider sub-Riemannian manifolds $(M,E,g)$ such that $M$ has dimension $2n+1$, while $E$ has rank $2n$. We say that it is a contact manifold if the mapping $\wedge^2 E \to TM/E$, $X,Y \mapsto [X,Y] \mod E$ is a non-degenerate skew-symmetric map. For such a sub-Riemannian manifold, it follows that $\gr_{x,-}$ is isomorphic to some $\mathfrak{g}_{-}^{\blambda}$ for any $x \in M$. The manifold has constant symbol if $\gr_{x,-}$ is isomorphic to the same $\mathfrak{g}_{-}^{\blambda}$ for every $x \in M$. Working locally, we can assume that $E$ is orientable. Let $\beta$ be the unique one-form, up to a sign, such that $\ker \beta=  E$, and
\begin{equation} \label{NormContactForm} \sup_{\begin{subarray}{c} v,w \in E_x \\ |v|=|w| =1
\end{subarray}} d\beta(v,w) =1, \qquad \text{for every $x \in M$}.\end{equation}
By definition we can write $d\beta(v,w) = \langle v, \Lambda J w \rangle_g$, $v,w \in E$, for which $J^* = -J$, $J^2 = -\id_E$ and where $\Lambda$ is a symmetric matrix with distinct eigenvalues $\lambda[1] =1 > \lambda[2] > \cdots >\lambda[l]$. Let $E[j]$ be the eigenspace of $\lambda[j]$ and write its dimension as~$2k[j]$.

\subsection{Conditions for curvature}
For $j=1,\dots, l$ and $r=1, \dots, k[j]$, write $e_r[j]$ for the standard basis of $\mathbb{C}^{k[j]}$. With slight abuse of notation, define
$$\tilde \chi =  \sum_{j=1}^l  \sum_{r=1}^{k[j]} \lambda[j] e_r[j] \wedge ie_r[j] \in \wedge^2 \mathfrak{g}_{-1},$$
so that
$$\partial_b = - \tilde \chi^* \wedge \iota_B, \qquad \partial_b^{-1} = - \frac{1}{|\tilde \chi|^2} B^* \wedge \iota_{\tilde \chi} .$$
We observe that $\partial_b \frakC^1_{\geq 1} = \spn\{ A \wedge \tilde \chi^*  \, : \, A \in \frakg_{-1} \oplus \frakg_0  \} $. We further see that
\begin{align*}
\partial \frakC_1^1 & = \spn \{ [s,\id] \wedge A_1^*   , A_2 \wedge \chi^* + [A_2,\id] \wedge B^*  \, : \, s \in \frakg_0, A_1, A_2 \in \frakg_{-1} \} \\
& = \spn \{ [s,\id] \wedge A_1^*    , [A_2, \id] \wedge B^*  \, : \, s \in \frakg_0, A_1, A_2 \in \frakg_{-1} \}  \mod \partial_b \frakC^1.
\end{align*}
Note that $\partial_b^{-1} \partial_b$ is the identity on elements $B^* \wedge [A_2, \id] $, so
\begin{align*}
(\id- \partial_b^{-1} \partial_b) \partial \frakC_1^1 & = \{  [s,\id] \wedge A_1^* \, : \, s \in \frakg_0, A_1 \in \frakg_{-1} \}.
\end{align*}
In conclusion, there is a unique connection satisfying
$$\kappa(\tilde \chi) = 0, \qquad \tr_{\frakg_{-1}} s\,\kappa(A, \cdot) =0, \qquad A \in \frakg_{-1}, s\in \frakg_0.$$
The latter is equivalent to
$$\langle A_1, \kappa(A_2, iA_3)\rangle = \langle- iA_1,  \kappa(A_2, A_3) \rangle , \qquad \langle A_1, \kappa(A_2, A_3) \rangle = \langle A_3, \kappa(A_2, A_1) \rangle,$$
with $A_2 \in \frakg_{-1}, A_1, A_3 \in \frakg_{-1}[j]$.

Seen from the manifold, we can define $\chi$ as a two-vector field such that if $X_1[j], \dots, X_{k[j]}[j], JX_1[j], \dots, JX_{k[j]}[j]$ is a local orthonormal basis of $E[j]$, then
$$\chi = \sum_{j=1}^l \sum_{r=1}^{k[j]} \lambda[j] X_{r}[j] \wedge JX_r[j].$$
We can verify that $\chi$ is independent of choice of basis, up to sign.
Conditions are that if $TM_{-2} = \spn \{ Z_{-2}\}$, then
$$R(\chi) =0, \qquad T(\chi) = \| \chi\|^2 Z_{-2}.$$
$$\langle v_1, T(v, Jv_2) \rangle = - \langle J v_1, T(v, v_2) \rangle, \qquad \langle v_1, T(v,v_2) \rangle = \langle v_2, T(v, v_1) \rangle.$$
for $v \in E$ and $v_1, v_2 \in E[j]$.

\subsection{Explicit description}
In order to give an explicit description of the connection, we first look at the case when the symbol is the standard Heisenberg algebra, i.e., when $\lambda_1 = \cdots =\lambda_n =1$.
Define the Reeb vector field $Z$ by the unique property that
\begin{equation} \label{Reeb} \beta(Z) = \|d\beta|_{\wedge^2 E} ||, \qquad d\beta(Z, \cdot) = 0.\end{equation}
In the case where the standard Heisenberg group is the symbol, we have $ \|d\beta|_{\wedge^2 E}\| = n$. Define $TM_{-2} = \spn \{ Z\}$ and let $g_I$ be the extension of the metric $g$
such that $Z$ becomes an orthonormal unit vector field. Let $\pr: TM \to E$ be the corresponding orthogonal projection to $E$ with kernel $TM_{-2}$. The space $\iso = \iso(E,g,I)$ then consist of maps
$s:TM \to TM$ such that $s(Z) = 0$, $s|_E$ is skew-symmetric with respect to the inner product and satisfying $sJ = Js$.

We can then define the partial connection
$$\tilde \nabla_X^E Y = \pr_E \nabla_X^{I} Y, \qquad X, Y \in \Gamma(E),$$
where $\nabla^{I}$ is the Levi-Civita connection of $g_I$. However, since this connection does not make $J$ parallel, we do not preserve our constant symbol structure, and we have to modify it
by defining
$$\nabla_X^E Y = \tilde \nabla_{X}Y + \frac{1}{2} (\tilde \nabla_X J) J.$$
We can extend this partial connection using the extension condition. We note that $T_1(X, Y) = \frac{1}{2}(\tilde \nabla_X J)J Y - \frac{1}{2}(\tilde \nabla_Y J)J X$, and since
$(\tilde \nabla_XJ) J$ is symmetric, $T_1$ will always be orthogonal to any $s \in \iso$. Furthermore, $T_1(X, JX)  = - \frac{1}{2} (\tilde \nabla_X J)X - \frac{1}{2} (\tilde \nabla_{JX} J)JX$, so we have
$$T_1(\chi) = \frac{1}{2n} \tr_{E} (\tilde \nabla _\times J)\times =0,$$
the latter equality proven in \cite[Lemma~7.1]{Gro20}. In summary, the connection $\nabla$ satisfies all the normalization conditions.

For the general case, we have a decomposition $E[1] \oplus \cdots \oplus E[l]$ corresponding to the eigenvalues of $d\beta$. Let $Z$ again denote the Reeb vector field defined by \eqref{Reeb},
where now $\|d\beta\| = \sum_{j=1}^l k[j] \cdot \lambda[j]$. For a general vector field $\Upsilon$ with values in $E$, define $TM_{-2} = \spn \{ Z - \Upsilon\}$ and let $\pr_\Upsilon: TM \to E$ denote the corresponding
projection to $E$ with $TM_{-2}$ as its kernel. Define an extended Riemannian metric $g_{\Upsilon}$ such that $Z-\Upsilon$ will be a unit vector orthogonal to $E$ with Levi-Civita connection~$\nabla^\Upsilon$. Write $\pr_\Upsilon[j]$ for the corresponding orthogonal projections to $E[j]$.
Finally, introduce the following operator
$$\langle \tau(X) Y, Y \rangle =  \frac{1}{2} \sum_{i=1}^l \sum_{i\neq j} (\calL_{\pr_\Upsilon[i] X} g_\Upsilon)(\pr_\Upsilon[j] Y, \pr_\Upsilon[j] Y_2), \qquad X, Y, Y_2 \in \Gamma(E).$$
Observe that the above expression is tensorial in all arguments, and that $\tau(X)$ will be a symmetric map such that $\tau(X) E[j] \subseteq E[j]$ for all $j=1, \dots, l$. We introduce now a partial connection
$$\text{for $X \in \Gamma(E[i])$ and $Y \in \Gamma(E[j])$}, \qquad
\tilde \nabla_X^E Y = \left\{ \begin{array}{ll} \pr_\Upsilon[j] \nabla_X^\Upsilon Y & \text{if $i = j$,} \\
\pr_\Upsilon[j] [X,Y] + \tau(X)Y & \text{if $i \neq j$.} \end{array} \right.$$
This partial connection is compatible with the metric and makes $E[j]$ parallel for each $j =1, \dots, l$. We modify it to make $J$ parallel in the same way as earlier $\nabla_X Y = \tilde \nabla_X Y + \frac{1}{2} (\tilde \nabla_X J) JY$.
It has been shown in \cite[Lemma~7.1]{Gro20} that $\tr_{E[j]} \tilde \nabla_{\times} J \times =0$, which gives us that
$$-T(\chi) = \| \chi \|^2 (Z - \Upsilon) + \frac{1}{2} \sum_{i=1}^l \sum_{i \neq j}\lambda[i] \tr_{E[i]} \pr_\Upsilon[j] [\times, J\times].$$
There exists a unique choice of $\Upsilon$ so that the second term vanishes which gives us that the condition $T_1(\chi) = \pr_{\Upsilon} T(\chi)=0$ holds. Explicitly, if we use that the latter condition is equivalent to $\pr_0 \pr_\Upsilon T(\chi) = 0$ and that
$$\pr_\Upsilon[j] [X,JX] = \pr_0[j] \pr_{\Upsilon} [X,JX] = \pr_0[j] [X,JX]- \lambda[i] \pr_0[j] \Upsilon .$$ for any unit vector field $X$ in $E[i]$,
\begin{align*}
0 &= \frac{1}{2} \sum_{i=1}^l \sum_{i \neq j}\lambda[i] \tr_{E[i]} \pr_\Upsilon[j] [\times, J\times] \\
& = \frac{1}{2} \sum_{i=1}^l \sum_{i \neq j}\lambda[i] \tr_{E[i]} \pr_0[j] [\times, J\times] - \sum_{i=1}^l \sum_{i \neq j}\lambda[i]^2 k[i] \pr_0[j] \Upsilon \\
& = \sum_{j=1}^l \left(\frac{1}{2}  \sum_{i \neq j}\lambda[i] \tr_{E[i]} \pr_0[j] [\times, J\times] - (\|\chi\|^2 - \lambda[j]^2 k[j]) \pr_0[j] \Upsilon \right).
\end{align*}
Hence, for $l = 1$, we use $\Upsilon =0$ and otherwise
$$\Upsilon =  \frac{1}{2} \sum_{j=1}^l \sum_{i \neq j} \frac{\lambda[i]}{\|\chi\|^2 - \lambda[j]^2 k[j]}  \tr_{E[i]} \pr_0[j] [\times, J\times].$$

\section{Maximal growth vector} \label{sec:MaxGrowth}
\subsection{The free Carnot algebra}
Let $\mathfrak{g}_{-1}$ be an inner product space and define
$$\mathfrak{g}_{-2} = \{ [X, Y] = X \otimes Y - Y \otimes X\, : \, X, Y \in \frakg_{-1} \}.$$
We avoid the use of the wedge product symbol not to confuse it with the wedge product between elements in $\frakC$. The Carnot algebra is called the free nilpotent algebra of step 2.
If $n_1> 1$ is the rank of $\frakg_{-1}$, then the growth vector is $(n_1, n_1(n+1)/2)$, with $n_1(n+1)/2$ being the maximal value that the growth vector can have. This Carnot algebra also have the largest possible size of $\frakg_0$ with rank $n_1(n_1-1)/2$.
If $S: \mathfrak{g}_{-1} \to \mathfrak{g}_{-1}$ is a linear isometry, then we can extend it to a Carnot algebra isometry by letting it work on $\frakg_{-2}$ by $S[X,Y] = [SX, SY]$. It follows that elements in $\frakg_0$ are skew-symmetric maps $s: \frakg_{-1} \to \frakg_{-1}$ extended by $s[X,Y] = [sX, Y] + [X,sY]$. We remark finally that $s \in \frakg_0$ applied to $\frakg_{-1}$ is spanned by elements of the form $A^*_1 \otimes A_2- A_2^* \otimes A_1$.

\subsection{Canonical condition for the curvature}
If $(M, E, g)$ is a sub-Riemannian manifold with growth vector $(n_1, n_1(n_1+1)/2)$, which then will have constant symbol given by the free nilpotent Lie algebra. Let $\psi$ be an adapted Cartan connection on the non-holonomic frame bundle $\scrF$, and with curvature $\kappa$. We want to express explicitly the extension and normalization condition for such a connection.

We compute the spaces of linear forms
$$\partial_b \frakC_{\geq 1}^1 = \spn\{  (s+B) \wedge A_1^* \wedge A_2^*  \, : \, A_1, A_2, B \in \frakg_{-1}, s \in
\frakg_0 \},$$
giving us the restriction
$$\kappa(A_1, A_2)  =0.$$
Next,
\begin{align*}
\partial &\frakC_{1}^1  = \spn \{ B \wedge A^* \wedge A_2^*  + [B, \id] \wedge [A_1, A_2]^* , [s,\id] \wedge A_1^*   \, : \, A_1, A_2, B \in \frakg_{-1}, s \in \frakg_0 \}  \\
& = \spn \{[B, \id]   \wedge[A_1, A_2]^* , [s,\pr_{-2} ] \wedge A^*  \, : \, A_1, A_2, B \in \frakg_{-1} , s \in \frakg_0\} \mod \partial_b \frakC_{\geq -1}^1.
\end{align*}
We observe that
$$(\id - \partial_b^{-1} \partial_b) [B_1,B_2]^* \wedge [A_1, A_2]^* = 0, $$
$$(\id - \partial_b^{-1} \partial_b) A_3^* \wedge [A_1, A_2]^* = A_3^* \wedge [A_1, A_2]^* - \frac{1}{3} \circlearrowright_{1,2,3} A_1^* \wedge [A_2, A_3]^*,$$
but since $\partial \kappa_1 =0$,
$$\langle \kappa_1 , \circlearrowright_{1,2,3} A_1^* \wedge [A_2, A_3]^* \rangle =  \circlearrowright_{1,2,3} \kappa(A_1, [A_2, A_3]) = \circlearrowright_{1,2,3} [A_1, \kappa(A_2, A_3)] =0.$$
Hence, the normalization conditions are, for $A_1, A_2, A_3 \in \frakg_{-1}, s \in \frakg_0$,
$$\tr_{\frakg_{-1}} \langle \kappa([A_1, A_2], \times), [A_3, \times] \rangle =0, \quad \tr_{\frakg_{-2}} \langle \kappa(A_1, \times), s\times \rangle=0, .$$
The latter can be written as
$$\tr_{\frakg_{-1}} \langle \kappa(A_1, [A_2, \times]), [A_3, \times] \rangle=\tr_{\frakg_{-1}} \langle \kappa(A_1, [A_3, \times]), [A_2, \times] \rangle.$$

\subsection{On the manifold}
We consider any sub-Riemannian manifold $(M, E,g)$ with growth vector $(n_1, n_1(n_1+1)/2)$. Since $\Gr_0|_E = \SO(E)$, any compatible partial connection $\nabla^E$ on $E$ can be extended to a strongly compatible connection.
Let $X_1, X_2, \dots, X_{n_1}$ be an orthonormal basis of $E$. We define a compatible connection
\begin{equation} \label{nablaFreeMu}\nabla_{X_k} X_j  = \nabla_{X_k}^E X_j= \sum_{i=1}^{n_1} \mu_{ij;k} X_i,  \qquad \mu_{ij;k} = - \mu_{ji;k},\end{equation}
for some functions $\mu_{ij;k}$.
Define also $TM_{-2} = \spn \{ Z_{ij}\}$ by
$$Z_{ij} = [X_i,X_j] +\sum_{r=1}^{n_1} (\mu_{ri;j} - \mu_{rj;i}) X_r,$$
and thus, $T_1(X_i,X_j) =0$.
By the extension condition
$$\nabla_{X_k} Z_{ij} = \sum_{r=1}^{n_1} (\mu_{ri;k} Z_{rj} - \mu_{rj;k} Z_{ri}), \qquad \nabla_{Z_{ij}} X_k = \sum_{r=1}^{n_1} (X_i \mu_{rk;j} - X_j\mu_{rk;i}) X_r.$$
Write coefficients $\nu_{ijk,rs}$ such that
$$[X_i, [X_j, X_k]] = \sum_{p<q} \nu_{ijk,pq} [X_p, X_q] \mod E, \qquad \nu_{ijk,pq} = -\nu_{ijk,pq},$$
which will also satisfy $\nu_{ijk;pq} = -\nu_{ikj;pq}$ and $\circlearrowright_{ijk} \nu_{ijk;pq} =0$.
We will only consider $n_1 >2$, as the case $n_1=2$ is in Section~\ref{sec:3dHeisenberg}.
\begin{proposition}
Define coefficients
$$\nu_{ij;k}^{(2)} = \sum_{r=1}^{n_1} \nu_{ijr;kr}.$$
Assume that $\nabla$ is generated by \eqref{nablaFreeMu} using the extension condition. Assume furthermore that it satisfies the normalization condition. Then for $i$, $j$, $k$ all different,
$$\mu_{ij;k} = \frac{1}{2(n_1-1)}(\nu^{(2)}_{kj;i}-\nu_{ki;j}^{(2)} ) + \frac{1}{2n_1(n_1-1)} (\nu_{ji;k}^{(2)} -\nu_{ij;k}^{(2)}).$$
Furthermore
\begin{enumerate}[\rm (a)]
\item if $n_1 > 3$, then for $i \neq j$
$$\mu_{ij;j} = \frac{1}{n_1-3} (\nu_{ij;j}^{(2)} - \nu_{jj;i}^{(2)}).$$
\item If $n_1 =3$, then for $i$, $j$, $k$ all different, taking all possible values, then
$$\mu_{ij;j} = \frac{2}{3}(\nu_{ij;j}^{(2)} - \nu_{ji;j}^{(2)}) - \frac{1}{3}(\nu_{ik;k}^{(2)} - \nu_{ki;k}^{(2)}).$$
\end{enumerate}
\end{proposition}
\begin{proof}
We first make the observation that
\begin{align*}
    \langle T(Z_{ij}&, X_k), Z_{pq} \rangle = \langle [X_k, Z_{ij}] - \nabla_{X_k} Z_{ij}, Z_{pq} \rangle \\
    & = \nu_{kij;pq} + (\mu_{qi;j}  -\mu_{qj;i})\delta_{kp} - (\mu_{pi;j} -\mu_{pj;i}) \delta_{kq} \\
    & \qquad - (\mu_{pi;k} \delta_{qj} - \mu_{pj;k} \delta_{qi}) + (\mu_{qi;k} \delta_{pj} - \mu_{qj;k} \delta_{pi}).
\end{align*}

The normalization conditions are then
\begin{align*}
0  &= \tr_E \langle T(T_0(X_i, X_j), \times), T_0(X_k, \times) \rangle  = \sum_{r=1}^{n_1} \langle T(Z_{ij}, X_r), Z_{kr} \rangle \\
& = \sum_{r=1}^{n_1}\nu_{rij;kr} - n_1 (\mu_{ki;j} -\mu_{kj;i})  + \sum_{r=1}^{n_1}(\mu_{ri;r} \delta_{kj} - \mu_{rj;r} \delta_{ki}) \\
& = -\nu_{ij;k}^{(2)} + \nu_{ji;k}^{(2)} - n_1 (\mu_{ki;j} -\mu_{kj;i})  + \sum_{r=1}^{n_1}(\mu_{ri;r} \delta_{kj} - \mu_{rj;r} \delta_{ki}).\end{align*}
Recall that these terms are skew-symmetric in $i,j$. For $j=k \neq i$, we have
\begin{align} \label{Norm1Free}
0  & = -\nu_{ij;j}^{(2)} + \nu_{ji;j}^{(2)} + n_1 \mu_{ij;j}  + \sum_{r=1}^{n_1}\mu_{ri;r},\end{align}
while for $i, j,k$ all different
\begin{align} \label{Norm2Free}
0  
& = -\nu_{ij;k}^{(2)} + \nu_{ji;k}^{(2)} - n_1 (\mu_{ki;j} -\mu_{kj;i}) .\end{align}
We also have normalization conditions, also skew-symmetric $i$ and $j$,
\begin{align*}
0 & = \tr_{E} \langle T(X_k, T_0(X_j, \times)), T_0(X_i, \times) \rangle -\tr_{\frakg_{-1}} \langle T(X_k, T_0(X_i, \times)), T_0(X_j, \times) \rangle \\
& = \sum_{r=1}^{n_1} \langle T(Z_{ir}, X_k), Z_{jr} \rangle - \sum_{r=1}^{n_1} \langle T(Z_{jr}, X_k), Z_{ir} \rangle \\ 
   & = \sum_{r=1}^{n_1} (\nu_{kir;jr} - \nu_{kjr;ir}) + 2 (n_1-1) \mu_{ij;k} +\mu_{jk;i} - \mu_{ik;j} + \sum_{r=1}^{n_1}(\mu_{ri;r}\delta_{kj} - \mu_{rj;r}\delta_{ki} ) 
        \end{align*}
which for $k=j \neq i $ equals
\begin{align} \label{Norm3Free}
0  & = \nu_{ji;j}^{(2)} - \nu_{jj;i}^{(2)} +  (2n_1-3) \mu_{ij;j}  + \sum_{r=1}^{n_1}\mu_{ri;r}  ,
     \end{align}
and for $i,j,k$ all different
\begin{align} \label{Norm4Free}
0  & = \nu_{ki;j}^{(2)} - \nu_{kj;i}^{(2)} + 2 (n_1-1) \mu_{ij;k} +\mu_{jk;i} - \mu_{ik;j} 
\end{align}

By combining \eqref{Norm2Free} and \eqref{Norm4Free} for $i$,$j$ and $k$ different, we get
$$\mu_{ij;k} = \frac{1}{2(n_1-1)}(\nu^{(2)}_{kj;i}-\nu_{ki;j}^{(2)} ) + \frac{1}{2n_1(n_1-1)} (\nu_{ji;k}^{(2)} -\nu_{ij;k}^{(2)}).$$
For $n_1 > 3$, we can subtract \eqref{Norm1Free} from \eqref{Norm3Free} to obtain that
$$\mu_{ij;j} = \frac{1}{n_1-3} (\nu_{ij;j}^{(2)} - \nu_{jj;i}^{(2)}).$$
For $n_1 = 3$, from \eqref{Norm1Free} we can insert respectively $(i,j)$ equal to $(1,2)$ and $(1,3)$, which equals
\begin{align*} 
2 \mu_{12;2}  - \mu_{13;3}  & = \nu_{12;2}^{(2)} - \nu_{21;2}^{(2)} \\
2 \mu_{13;3}  - \mu_{12;2}  & = \nu_{13;3}^{(2)} - \nu_{31;3}^{(2)},\end{align*}
which imply that
$$3 \mu_{12;2} = 2(\nu_{12;2}^{(2)} - \nu_{21;2}^{(2)}) - \nu_{13;3}^{(2)} + \nu_{31;3}^{(2)}.$$
We can then finally use a similar approach to all other cases to get the result.
\end{proof}

\bibliographystyle{habbrv}
\bibliography{Bibliography}

@article {AMS19,
    AUTHOR = {Alekseevsky, D. and Medvedev, A. and Slovak, J.},
     TITLE = {Constant curvature models in sub-{R}iemannian geometry},
   JOURNAL = {J. Geom. Phys.},
  FJOURNAL = {Journal of Geometry and Physics},
    VOLUME = {138},
      YEAR = {2019},
     PAGES = {241--256},
      ISSN = {0393-0440},
   MRCLASS = {53C17 (17B70 58A15)},
  MRNUMBER = {3945041},
       DOI = {10.1016/j.geomphys.2018.09.013},
       URL = {https://doi.org/10.1016/j.geomphys.2018.09.013},
}

@article {Agr06,
    AUTHOR = {Agricola, Ilka},
     TITLE = {The {S}rn\'i\ lectures on non-integrable geometries with
              torsion},
   JOURNAL = {Arch. Math. (Brno)},
  FJOURNAL = {Universitatis Masarykianae Brunensis. Facultas Scientiarum
              Naturalium. Archivum Mathematicum},
    VOLUME = {42},
      YEAR = {2006},
     PAGES = {5--84},
      ISSN = {0044-8753,1212-5059},
   MRCLASS = {53C10 (53C25 53C27 53C29 53D15 58J60 81T30)},
  MRNUMBER = {2322400},
MRREVIEWER = {Anna\ M.\ Fino},
}

@article {ABR18,
    AUTHOR = {Agrachev, A. and Barilari, D. and Rizzi, L.},
     TITLE = {Curvature: a variational approach},
   JOURNAL = {Mem. Amer. Math. Soc.},
  FJOURNAL = {Memoirs of the American Mathematical Society},
    VOLUME = {256},
      YEAR = {2018},
    NUMBER = {1225},
     PAGES = {v+142},
      ISSN = {0065-9266},
      ISBN = {978-1-4704-2946-0; 978-1-4704-4913-1},
   MRCLASS = {49-02 (53C17 58B20)},
  MRNUMBER = {3852258},
}

@book {CaSl09,
    AUTHOR = {{\v C}ap, Andreas and Slov\'{a}k, Jan},
     TITLE = {Parabolic geometries. {I}},
    SERIES = {Mathematical Surveys and Monographs},
    VOLUME = {154},
      NOTE = {Background and general theory},
 PUBLISHER = {American Mathematical Society, Providence, RI},
      YEAR = {2009},
     PAGES = {x+628},
      ISBN = {978-0-8218-2681-2},
   MRCLASS = {53C10 (17B60 22E60 53A55 53C30 58J70)},
  MRNUMBER = {2532439},
MRREVIEWER = {Stuart Armstrong},
       DOI = {10.1090/surv/154},
       URL = {https://doi.org/10.1090/surv/154},
}

@article {Gro22,
    AUTHOR = {Grong, Erlend},
     TITLE = {Curvature and the equivalence problem in sub-{R}iemannian
              geometry},
   JOURNAL = {Arch. Math. (Brno)},
  FJOURNAL = {Universitatis Masarykianae Brunensis. Facultas Scientiarum
              Naturalium. Archivum Mathematicum},
    VOLUME = {58},
      YEAR = {2022},
    NUMBER = {5},
     PAGES = {295--327},
      ISSN = {0044-8753,1212-5059},
   MRCLASS = {53C17 (58A15)},
  MRNUMBER = {4529821},
}

@article {FGR97,
    AUTHOR = {Falbel, Elisha and Gorodski, Claudio and Rumin, Michel},
     TITLE = {Holonomy of sub-{R}iemannian manifolds},
   JOURNAL = {Internat. J. Math.},
  FJOURNAL = {International Journal of Mathematics},
    VOLUME = {8},
      YEAR = {1997},
    NUMBER = {3},
     PAGES = {317--344},
      ISSN = {0129-167X},
   MRCLASS = {53C05},
  MRNUMBER = {1454476 (98e:53032)},
MRREVIEWER = {Fernand Pelletier},
       DOI = {10.1142/S0129167X97000159},
       URL = {http://dx.doi.org/10.1142/S0129167X97000159},
}

@ARTICLE{Gro20,
       author = {{Grong}, Erlend},
        title = "{Canonical connections on sub-Riemannian manifolds with constant symbol}",
      journal = {arXiv e-prints},
         year = 2020,
        month = oct,
          eid = {arXiv:2010.05366},
        pages = {arXiv:2010.05366},
          doi = {10.48550/arXiv.2010.05366},
archivePrefix = {arXiv},
       eprint = {2010.05366},
 primaryClass = {math.DG},
       adsurl = {https://ui.adsabs.harvard.edu/abs/2020arXiv201005366G}
}

@ARTICLE{GrTr23,
       author = {{Grong}, Erlend and {Tripaldi}, Francesca},
        title = "{Filtered complexes and cohomologically equivalent subcomplexes}",
      journal = {arXiv e-prints},
         year = 2023,
        month = aug,
          eid = {arXiv:2308.11353},
        pages = {arXiv:2308.11353},
          doi = {10.48550/arXiv.2308.11353},
archivePrefix = {arXiv},
       eprint = {2308.11353},
 primaryClass = {math.DG},
       adsurl = {https://ui.adsabs.harvard.edu/abs/2023arXiv230811353G}
}

@article {LiZe09,
    AUTHOR = {Zelenko, Igor and Li, Chengbo},
     TITLE = {Differential geometry of curves in {L}agrange {G}rassmannians
              with given {Y}oung diagram},
   JOURNAL = {Differential Geom. Appl.},
  FJOURNAL = {Differential Geometry and its Applications},
    VOLUME = {27},
      YEAR = {2009},
    NUMBER = {6},
     PAGES = {723--742},
      ISSN = {0926-2245},
     CODEN = {DGAPEO},
   MRCLASS = {53D12 (53C40)},
  MRNUMBER = {2552681 (2010i:53159)},
MRREVIEWER = {Valentin Ovsienko},
       DOI = {10.1016/j.difgeo.2009.07.002},
       URL = {http://dx.doi.org/10.1016/j.difgeo.2009.07.002},
}

@article {Mor93,
    AUTHOR = {Morimoto, Tohru},
     TITLE = {Geometric structures on filtered manifolds},
   JOURNAL = {Hokkaido Math. J.},
  FJOURNAL = {Hokkaido Mathematical Journal},
    VOLUME = {22},
      YEAR = {1993},
    NUMBER = {3},
     PAGES = {263--347},
      ISSN = {0385-4035},
   MRCLASS = {58H05 (53C10)},
  MRNUMBER = {1245130},
MRREVIEWER = {J. Chrastina},
       DOI = {10.14492/hokmj/1381413178},
       URL = {https://doi.org/10.14492/hokmj/1381413178},
}

@article {Mor08,
    AUTHOR = {Morimoto, Tohru},
     TITLE = {Cartan connection associated with a subriemannian structure},
   JOURNAL = {Differential Geom. Appl.},
  FJOURNAL = {Differential Geometry and its Applications},
    VOLUME = {26},
      YEAR = {2008},
    NUMBER = {1},
     PAGES = {75--78},
      ISSN = {0926-2245},
   MRCLASS = {53C17},
  MRNUMBER = {2393974},
MRREVIEWER = {Luis Guijarro},
       DOI = {10.1016/j.difgeo.2007.12.002},
       URL = {https://doi.org/10.1016/j.difgeo.2007.12.002},
}

@article {OR12,
    AUTHOR = {Olmos, Carlos and Reggiani, Silvio},
     TITLE = {The skew-torsion holonomy theorem and naturally reductive
              spaces},
   JOURNAL = {J. Reine Angew. Math.},
  FJOURNAL = {Journal f\"ur die Reine und Angewandte Mathematik. [Crelle's
              Journal]},
    VOLUME = {664},
      YEAR = {2012},
     PAGES = {29--53},
      ISSN = {0075-4102,1435-5345},
   MRCLASS = {53C30 (53C29 53C35)},
  MRNUMBER = {2980129},
MRREVIEWER = {Sergio\ Console},
       DOI = {10.1515/crelle.2011.100},
       URL = {https://doi.org/10.1515/crelle.2011.100},
}

@article {Rum90,
    AUTHOR = {Rumin, Michel},
     TITLE = {Un complexe de formes diff\'erentielles sur les vari\'et\'es
              de contact},
   JOURNAL = {C. R. Acad. Sci. Paris S\'er. I Math.},
  FJOURNAL = {Comptes Rendus de l'Acad\'emie des Sciences. S\'erie I.
              Math\'ematique},
    VOLUME = {310},
      YEAR = {1990},
    NUMBER = {6},
     PAGES = {401--404},
      ISSN = {0764-4442},
     CODEN = {CASMEI},
   MRCLASS = {58A10 (32F25 53C15)},
  MRNUMBER = {1046521 (91a:58004)},
MRREVIEWER = {M. Kalka},
}

@article {Rum94,
    AUTHOR = {Rumin, Michel},
     TITLE = {Formes diff\'erentielles sur les vari\'et\'es de contact},
   JOURNAL = {J. Differential Geom.},
  FJOURNAL = {Journal of Differential Geometry},
    VOLUME = {39},
      YEAR = {1994},
    NUMBER = {2},
     PAGES = {281--330},
      ISSN = {0022-040X},
     CODEN = {JDGEAS},
   MRCLASS = {58G05 (53C15 58A10 58E10)},
  MRNUMBER = {1267892 (95g:58221)},
MRREVIEWER = {Antonella Nannicini},
       URL = {http://projecteuclid.org/euclid.jdg/1214454873},
}

@article {Rum99,
    AUTHOR = {Rumin, Michel},
     TITLE = {Differential geometry on {C}-{C} spaces and application to the
              {N}ovikov-{S}hubin numbers of nilpotent {L}ie groups},
   JOURNAL = {C. R. Acad. Sci. Paris S\'{e}r. I Math.},
  FJOURNAL = {Comptes Rendus de l'Acad\'{e}mie des Sciences. S\'{e}rie I.
              Math\'{e}matique},
    VOLUME = {329},
      YEAR = {1999},
    NUMBER = {11},
     PAGES = {985--990},
      ISSN = {0764-4442},
   MRCLASS = {53C17 (58H99 58J10)},
  MRNUMBER = {1733906},
MRREVIEWER = {Fernand Pelletier},
       DOI = {10.1016/S0764-4442(00)88624-3},
       URL = {https://doi.org/10.1016/S0764-4442(00)88624-3},
}

@article {CGJK19,
    AUTHOR = {Chitour, Yacine and Grong, Erlend and Jean, Fr\'{e}d\'{e}ric and
              Kokkonen, Petri},
     TITLE = {Horizontal holonomy and foliated manifolds},
   JOURNAL = {Ann. Inst. Fourier (Grenoble)},
  FJOURNAL = {Universit\'{e} de Grenoble. Annales de l'Institut Fourier},
    VOLUME = {69},
      YEAR = {2019},
    NUMBER = {3},
     PAGES = {1047--1086},
      ISSN = {0373-0956},
   MRCLASS = {53C29 (53C12)},
  MRNUMBER = {3986910},
       URL = {http://aif.cedram.org/item?id=AIF_2019__69_3_1047_0},
}

@article {LDO16,
    AUTHOR = {Le Donne, Enrico and Ottazzi, Alessandro},
     TITLE = {Isometries of {C}arnot groups and sub-{F}insler homogeneous
              manifolds},
   JOURNAL = {J. Geom. Anal.},
  FJOURNAL = {Journal of Geometric Analysis},
    VOLUME = {26},
      YEAR = {2016},
    NUMBER = {1},
     PAGES = {330--345},
      ISSN = {1050-6926},
   MRCLASS = {53C17 (53C60 58D05)},
  MRNUMBER = {3441517},
MRREVIEWER = {Davide Vittone},
       DOI = {10.1007/s12220-014-9552-8},
       URL = {https://doi.org/10.1007/s12220-014-9552-8},
}

@article {Tan89,
    AUTHOR = {Tanno, Shukichi},
     TITLE = {Variational problems on contact {R}iemannian manifolds},
   JOURNAL = {Trans. Amer. Math. Soc.},
  FJOURNAL = {Transactions of the American Mathematical Society},
    VOLUME = {314},
      YEAR = {1989},
    NUMBER = {1},
     PAGES = {349--379},
      ISSN = {0002-9947},
   MRCLASS = {53C15 (32F25 58G30)},
  MRNUMBER = {1000553},
MRREVIEWER = {Gerhard Huisken},
       DOI = {10.2307/2001446},
       URL = {https://doi.org/10.2307/2001446},
}

@incollection {Bel96,
    AUTHOR = {Bella\"{\i}che, Andr\'{e}},
     TITLE = {The tangent space in sub-{R}iemannian geometry},
 BOOKTITLE = {Sub-{R}iemannian geometry},
    SERIES = {Progr. Math.},
    VOLUME = {144},
     PAGES = {1--78},
 PUBLISHER = {Birkh\"{a}user, Basel},
      YEAR = {1996},
   MRCLASS = {53C99 (57R27)},
  MRNUMBER = {1421822},
MRREVIEWER = {Claudio Gorodski},
       DOI = {10.1007/978-3-0348-9210-0_1},
       URL = {https://doi.org/10.1007/978-3-0348-9210-0_1},
}

@article {Zel09,
    AUTHOR = {Zelenko, Igor},
     TITLE = {On {T}anaka's prolongation procedure for filtered structures
              of constant type},
   JOURNAL = {SIGMA Symmetry Integrability Geom. Methods Appl.},
  FJOURNAL = {SIGMA. Symmetry, Integrability and Geometry. Methods and
              Applications},
    VOLUME = {5},
      YEAR = {2009},
     PAGES = {Paper 094, 21},
      ISSN = {1815-0659},
   MRCLASS = {58A30 (58A17)},
  MRNUMBER = {2559667},
MRREVIEWER = {Josef\ Jany\v ska},
       DOI = {10.3842/SIGMA.2009.094},
       URL = {https://doi.org/10.3842/SIGMA.2009.094},
}

@ARTICLE{Gro20b,
       author = {{Grong}, Erlend},
        title = "{Affine connections and curvature in sub-Riemannian geometry}",
      journal = {arXiv e-prints},
         year = "2020",
        month = "Jan",
          eid = {arXiv:2001.03817},
        pages = {arXiv:2001.03817},
archivePrefix = {arXiv},
       eprint = {2001.03817},
 primaryClass = {math.DG},
       adsurl = {https://ui.adsabs.harvard.edu/abs/2020arXiv200103817G}
}

@article {CapGrover14,
    AUTHOR = {{\v C}ap, A. and Gover, A. R. and Hammerl, M.},
     TITLE = {Holonomy reductions of {C}artan geometries and curved orbit
              decompositions},
   JOURNAL = {Duke Math. J.},
  FJOURNAL = {Duke Mathematical Journal},
    VOLUME = {163},
      YEAR = {2014},
    NUMBER = {5},
     PAGES = {1035--1070},
      ISSN = {0012-7094,1547-7398},
   MRCLASS = {53C29 (32Q20 32V05 53C25 53C55)},
  MRNUMBER = {3189437},
MRREVIEWER = {Vojt\v ech\ \v Z\'adn\'ik},
       DOI = {10.1215/00127094-2644793},
       URL = {https://doi.org/10.1215/00127094-2644793},
}

@article{hladky2009connections,
    AUTHOR = {Hladky, Robert K.},
     TITLE = {Connections and curvature in sub-{R}iemannian geometry},
   JOURNAL = {Houston J. Math.},
  FJOURNAL = {Houston Journal of Mathematics},
    VOLUME = {38},
      YEAR = {2012},
    NUMBER = {4},
     PAGES = {1107--1134},
      ISSN = {0362-1588},
   MRCLASS = {53C17 (53C05)},
  MRNUMBER = {3019025},
MRREVIEWER = {L\'eonard\ Todjihounde},
}

\end{document}